\documentclass[12pt,reqno]{amsart}%
\usepackage{amsfonts}
\usepackage{amsmath}
\usepackage{amssymb}
\usepackage{color}
\usepackage{graphicx}%
\providecommand{\U}[1]{\protect\rule{.1in}{.1in}}
\newtheorem{theorem}{Theorem}[section]
\theoremstyle{plain}

\newtheorem{lemma}{Lemma}[section]

\newtheorem{remark}{Remark}

\numberwithin{equation}{section}
\begin{document}
\title[Stability for the critical points of the HLS inequality]{Stability for Critical Points of the Hardy--Littlewood--Sobolev Inequality and a Dual Stability Framework}
\author{Lu Chen}
\address[Lu Chen]{Key Laboratory of Algebraic Lie Theory and Analysis of Ministry of Education, School of Mathematics and Statistics, Beijing Institute of Technology, Beijing
100081, PR China}
\email{chenlu5818804@163.com}

\author{Guozhen Lu}
\address[Guozhen Lu]{Department of Mathematics, University of Connecticut, Storrs, CT 06269, USA}
\email{guozhen.lu@uconn.edu}

\author{Hanli Tang}
\address[Hanli Tang]{Laboratory of Mathematics and Complex Systems (Ministry of Education), School of Mathematical Sciences, Beijing Normal University, Beijing, 100875, China}
\email{hltang@bnu.edu.cn}
\address{}
\thanks{The first author was partly supported by the National Key Research and Development Program (No.
2022YFA1006900) and National Natural Science Foundation of China (No. 12271027); the second author was partly supported by grants from Simons Foundation; and the third author
was partly supported by  National Natural Science Foundation of China (Grant No.12471100).}


\begin{abstract}
 Although quantitative stability for critical points of the Sobolev and fractional Sobolev inequalities has been extensively studied, the corresponding stability theory for critical points of the Hardy--Littlewood--Sobolev (HLS) inequality remains largely unexplored. A major difficulty is that the natural stability problem for HLS critical points involves a non-Hilbertian distance, so the classical orthogonal decomposition methods used in Hilbert-space settings are no longer available.

 In this paper, we develop a weak-decomposition--strong-stability method tailored to the stability structure of HLS critical points and establish the corresponding stability inequality. Our approach also yields an explicit lower bound for the stability of Palais--Smale sequences of the HLS integral equation. To the best of our knowledge, this appears to be the first quantitative stability result for Palais--Smale sequences of a variational functional measured in a non-Hilbertian distance. We further introduce a duality framework connecting Struwe-type decompositions and stability inequalities for critical points of the Sobolev inequality with their HLS counterparts. As a consequence, we derive Struwe-type decomposition and stability results for critical points of the fractional Sobolev inequality for general functions, thereby removing the nonnegativity assumption imposed in \cite{NK}.

\end{abstract}
\maketitle
\section{Introduction}
\subsection{Background}
Lieb's sharp form \cite{Lieb} of the Hardy-Littlewood-Sobolev inequality in $\mathbb{R}^n$  for $0<s<\frac {n}{2}$  states
\begin{equation}
	\label{eq-hls}
	\| f\|_{L^{\frac{2n}{n+2s}}(\mathbb{R}^n)}^2\geq \mathcal S_{s,n}\left \|(-\Delta)^{-s/2} f \right\|_{L^2{(\mathbb{R}^n)}}^2
	\qquad\text{for all}\ f\in L^{\frac{2n}{n+2s}}(\mathbb{R}^n)
\end{equation}
with
\begin{equation}
	\label{eq:sobconst}
	\mathcal S_{s,n} = (4\pi)^s \ \frac{\Gamma(\frac{n+2s}{2})}{\Gamma(\frac{n-2s}{2})} \left( \frac{\Gamma(\frac n2)}{\Gamma(n)} \right)^{2s/n}
	= \frac{\Gamma(\frac{n+2s}{2})}{\Gamma(\frac{n-2s}{2})} \ |\mathbb{S}^n|^{2s/n} \,
\end{equation}
being the sharp constant. Furthermore, he also showed that the equality of HLS inequality \eqref{eq-hls} holds if and only if
$$f\in M_{HLS}:=\left\{cF(\frac{\cdot-x_0}{a}): c\in \mathbb{R},\  x_0\in \mathbb{R}^n,\  a>0 \right\},$$
where $F(x)=(1+|x|^2)^{-\frac{n+2s}{2}}$.
\medskip

The Euler-Lagrange equation of the optimization problem related to (\ref{eq-hls}) reads, in a suitable normalization,
\begin{equation}\label{hls-equation}
(-\Delta)^{-s}f=|f|^{-\frac{4s}{n+2s}}f.
\end{equation}
The classification of the positive solution of (\ref{hls-equation}) was solved by Chen, Li and Ou \cite{CLO} and Li \cite{Li}. They show that a positive solution $f$ satisfies (\ref{hls-equation}) if and only if
$$f\in \mathcal{M}_{HLS}^{+}:=\left\{c_{n,s}\left(\frac{\lambda}{1+\lambda^2|x-z|^2}\right)^{\frac{n+2s}{2}}:   z\in \mathbb{R}^n,\  \lambda>0 \right\},$$
where $$c_{n,s}=\left(\frac{\Gamma(\frac{n}{2}+s)}{\Gamma(\frac{n}{2}-s)}\right)^{\frac{n+2s}{4s}}.$$

\subsection{Stability problem of the Sobolev inequality}

Lieb in \cite{Lieb} also proved the sharp Sobolev inequality in $\mathbb{R}^n$ as an equivalent reformulation of the sharp Hardy-Littlewood-Sobolev inequality in the case of conformal index. Specifically, he proved for $0<s<\frac {n}{2}$ there holds
\begin{equation}
	\label{eq-sob}
	\left\|(-\Delta)^{s/2} g \right\|_{L^2(\mathbb{R}^n)}^2 \geq \mathcal S_{s,n} \| g\|_{L^{\frac{2n}{n-2s}}(\mathbb{R}^n)}^2
	\qquad\text{for all}\ g\in \dot H^s(\mathbb{R}^n),
\end{equation}
where $\dot H^s(\mathbb{R}^n)$ denotes the completion of $C_{c}^{\infty}(\mathbb{R}^n)$ under the norm $\big(\int_{\mathbb{R}^n}|(-\Delta)^{\frac{s}{2}}g|^2dx\big)^{\frac{1}{2}}$.
The sharp constant $\mathcal S_{s,n}$ of inequality (\ref{eq-sob}) has been computed first by Rosen \cite{Ro} in the case $s=1$, $n=3$ and then independently by Aubin \cite{Au} and Talenti \cite{Ta} in the case $s=1$. Furthermore, Lieb also showed that the equality of Sobolev inequality \eqref{eq-sob} holds if and only if
$$g\in M_s:=\left\{cV(\frac{\cdot-x_0}{a}): c\in \mathbb{R},\  x_0\in \mathbb{R}^n,\  a>0 \right\},$$
where $V(x)=(1+|x|^2)^{-\frac{n-2s}{2}}$.
\medskip

The Euler-Lagrange equation of the optimization problem related to (\ref{eq-sob}) reads, under a suitable normalization,
\begin{equation}\label{sob-equation}
(-\Delta)^{s}g=|g|^{\frac{4s}{n-2s}}g,
\end{equation}
which is equivalent to (\ref{hls-equation}). Thus a positive solution $g$ satisfies (\ref{sob-equation}) if and only if
$$g\in \mathcal{M}_s^{+}:=\left\{d_{n,s}\left(\frac{\lambda}{1+\lambda^2|x-z|^2}\right)^{\frac{n-2s}{2}}:   z\in \mathbb{R}^n,\  \lambda>0 \right\},$$
where $$d_{n,s}=\left(\frac{\Gamma(\frac{n}{2}+s)}{\Gamma(\frac{n}{2}-s)}\right)^{\frac{n-2s}{4s}}.$$
Denote
$$U(z,\lambda)(x)=d_{n,s}\left(\frac{\lambda}{1+\lambda^2|x-z|^2}\right)^{\frac{n-2s}{2}}\in \mathcal{M}_{s},$$
and we will call such functions Aubin-Talenti bubbles.
\medskip

Once the sharp inequality (\ref{eq-sob}) is established and solutions to (\ref{sob-equation}) are characterized, it is natural to ask the stability question. Roughly speaking, if
a function almost satisfies (\ref{sob-equation}) or (\ref{eq-sob}), in some sense, it must be close to an Aubin-Talenti bubble or an Aubin-Talenti bubble up to scaling. Nowadays,  these two kinds of stability problems are known as the stability of the Sobolev inequality and the stability of critical points of the Sobolev inequality respectively. More specifically, the stability of the Sobolev inequality started from the work of Brezis and Lieb. In \cite{BrLi} they asked if the following refined first order Sobolev inequality ($s=1$ in (\ref{eq-sob}))
holds for some distance function $d$:
$$\left\|(-\Delta)^{1/2} g \right\|_{L^2(\mathbb{R}^n)}^2 - \mathcal S_{1,n} \| g\|_{L^{\frac{2n}{n-2}}(\mathbb{R}^n)}^2\geq c d^{2}(g, M_1).$$
This question was answered affirmatively  in a pioneering work by Bianchi and Egnell \cite{BiEg}, subsequently in the case $s=2$ by the second author and Wei \cite{LuWe} and in the case of any positive even integer $s<n/2$ by
Bartsch, Weth and Willem \cite{BaWeWi}. In 2013, Chen, Frank and Weth \cite{ChFrWe} established the stability of the Sobolev inequality for all $0<s<n/2$. They proved that
\begin{equation}\label{Sob sta ine}
\left\|(-\Delta)^{s/2} g \right\|_{L^2(\mathbb{R}^n)}^2 - \mathcal S_{s,n} \| g\|_{L^{\frac{2n}{n-2s}}(\mathbb{R}^n)}^2\geq C_{n,s} d^{2}(g, M_s),
\end{equation}
for all $g\in \dot H^s(\mathbb{R}^n)$, where $C_{n,s}>0$ and $d(g,M_s)=\min\{\|(-\Delta)^{s/2}(U-\phi)\|_{L^2}:\phi \in M_s\}$.
\medskip

Recently, Dolbeault, Esteban, Figalli, Frank and Loss \cite{DEFFL} established the optimal quantitative stability of the sobolev inequality for the first time. And the current authors \cite{CLT1,CLT2,CLT3} set up the optimal stability of the fractional Sobolev inequality for all $0<s<n/2$. In summary, the results of \cite{CLT1, CLT2, CLT3,DEFFL} about the optimal stability of fractional Sobolev inequality can be stated as follows:

\vskip0.3cm
\textbf{Theorem A }
\textit{For any fixed} $s\in(0,n/2)$, \textit{there exists a positive constant} $\beta_s$ \textit{such that for any} $f\in  \dot H^s(\mathbb{R}^n)$, \textit{there holds}
$$\left\| (-\Delta)^{s/2} f \right\|_{L^2(\mathbb{R}^n)}^2-\mathcal S_{s,n} \|f\|_{L^{\frac{2n}{n-2s}}(\mathbb{R}^n)}^2\geq \frac{\beta_{s}}{n} \inf_{h\in M_s}\|(-\Delta)^{s/2}(f-h)\|_{L^2(\mathbb{R}^n)}^2.$$
Furthermore, $\beta_s$ behaviors like $s$ when $s\rightarrow 0$. We also mention that
Dolbeault, Esteban, Figalli, Frank and Loss \cite{DEFFL} established the optimal stability of Gross-type Log-Sobolev inequality in $\mathbb{R}^n$ \cite{Gr}  through dimension-limit method as the dimension goes to infinity and Chen, Lu, Tang \cite{CLT2} established the optimal stability of Beckner's Log-Sobolev inequality
in $\mathbb{S}^n$ obtained in \cite{Be1992,Be1997} through order-limit method as the order goes to zero. This sharpens the local stability of Log-Sobolev inequality in $\mathbb{S}^n$ obtained earlier in \cite{CLT-JFA}.
\medskip

On the other hand, study on the  stability problem  for critical points of the Sobolev inequality started from the celebrated work of Struwe \cite{Str} for $s=1$, then generalized by G\'{e}rard \cite{Ge} for $0<s<n/2$, see also  Palatucci and Pisante \cite[Theorem1.1]{PP} and Fang and Gonz\'{a}lez \cite[Theorem 1.3]{FG}. Here we only state the result for the one bubble case.
\vskip0.3cm
\textbf{Theorem B. }
\textit{Let} $0<s<n/2$ \textit{and} $\{u_k\}$ \textit{be a sequence of nonnegative functions such that}
$$\frac{1}{2}\mathcal{S}_{s,n}^{\frac{n}{2s}}\leq \|u_k\|^2_{\dot H^s(\mathbb{R}^n)}\leq \frac{3}{2}\mathcal{S}_{s,n}^{\frac{n}{2s}}.$$ \textit{If it satisfies}
$$\|(-\Delta)^{s}u_k-u_{k}^{\frac{n+2s}{n-2s}}\|_{\dot H^{-s}(\mathbb{R}^n)}\rightarrow 0,$$
\textit{Then there exist sequences} $\{z_k\}\subset \mathbb{R}^n$ \textit{and} $\lambda_k\subset (0,+\infty)$ \textit{such that}
$$\|u_k-U(z_k,\lambda_k)\|_{\dot H^s(\mathbb{R}^n)}\rightarrow 0 ~~~\textit{as}~~~~~ k\rightarrow \infty.$$
\vskip0.3cm

Ciraolo, Figalli and Maggi \cite{CFM} established the stability of critical points of Sobolev inequality with the single bubble. Later, Figalli and Glaudo \cite{FG} and Deng, Sun and Wei \cite{DSW} established the stability of critical points of Sobolev inequality with the multi-bubble when $3\leq n\leq 5$ and $n\geq 6$ respectively. De Nitti and K\"{o}nig \cite{NK} established the following quantitative stability for critical points of fractional Sobolev inequalities for non-negative function.
\vskip0.3cm
\textbf{Theorem C. }
\textit{For} $0<s<n/2$ \textit{and nonnegative} $u$ \textit{satisfying}
\begin{equation}\label{condition}
\frac{1}{2}\mathcal{S}_{s,n}^{\frac{n}{2s}}\leq \|u\|^2_{\dot H^s(\mathbb{R}^n)}\leq \frac{3}{2}\mathcal{S}_{s,n}^{\frac{n}{2s}}.
\end{equation}
 \textit{Then there exist a positive constant} $C_{cp}$ \textit{such that}
$$\|(-\Delta)^{s}u-u^{\frac{n+2s}{n-2s}}\|_{\dot H^{-s}(\mathbb{R}^n)}\geq C_{cp}\inf_{h\in \mathcal{M}_s^{+}}\|u-h\|_{\dot H^{s}(\mathbb{R}^n)}.$$ Furthermore, they also gave the sharp lower bound for $C_{cp}$ if $u$ is replaced by Palais-Smale sequence of fractional Sobolev equation.
\vskip0.3cm

\subsection{Stability problem of the Hardy-Littlewood-Sobolev inequality and our main results}
Although the HLS inequality (\ref{eq-hls}) is equivalent to the Sobolev inequality (\ref{eq-sob}) and stability of the Sobolev inequality has been
investigated broadly, the stability of the HLS inequality has been much less explored. Until 2015, Carlen \cite{Carlen} established the stability of the HLS inequality by a dual stability method. Due to the non-Hilbert distance $\|\cdot\|_{L^{\frac{2n}{n+2s}}}$, it seems difficult to prove the stability of the HLS inequality directly. Recently, the current authors \cite{CLT3} developed a method based on the $H^{-s}$-decomposition instead of $L^{\frac{2n}{n+2s}}$-decomposition to obtain the local stability of the HLS-inequality, then directly established the stability of the HLS-inequality with the optimal order and dimension-dependent constants.
\vskip0.3cm
\textbf{Theorem D}
\textit{For any fixed} $s\in(0,n/2)$, \textit{there exists a positive constant} $\beta_s$ \textit{such that for any} $f\in L^{\frac{2n}{n+2s}}(\mathbb{R}^n)$, \textit{there holds}
$$\|f\|^2_{L^{\frac{2n}{n+2s}}(\mathbb{R}^n)}-\mathcal{S}_{s,n}\|(-\Delta)^{-s/2}g\|^2_{L^2(\mathbb{R}^n)}\geq \frac{\beta_s}{n}\inf_{h\in M_{HLS}}\|f-h\|^2_{L^{\frac{2n}{n+2s}}(\mathbb{R}^n)}.$$
\textit{Furthermore, there exist two absolutely constants} $c,C>0$ \textit{such that} $c\leq \frac{\beta_s}{s}\leq C$ \textit{when} $s\rightarrow 0$.
\vskip0.3cm

Although the quantitative stability for the critical points of the Sobolev inequality has been studied extensively, the stability for critical points of the HLS inequality remains largely  unexplored. It is mainly because the stability for critical points of the HLS inequality involves the non-Hilbert distance, classical orthogonal decomposition technique to deal with Hilbert distance fails. The main purpose of this paper is to develop some new technique to deal with the stability problem for critical points of the HLS inequality. We will  prove the following

\vskip0.3cm
\begin{theorem}\label{stability}
Let $0<s<n/2$ and $f\in L^{\frac{2n}{n+2s}}(\mathbb{R}^n)$ whose energy satisfies
$$\frac{1}{2}\mathcal{S}_{s,n}^{\frac{n}{2s}}\leq\int_{\mathbb{R}^n}|f|^{\frac{2n}{n+2s}}dx\leq \frac{3}{2}\mathcal{S}_{s,n}^{\frac{n}{2s}}.$$
Then there exists a positive constant $C_{n,s}$ such that
$$\left\||f|^{-\frac{4s}{n+2s}}f-(-\Delta)^{-s}f\right\|_{L^\frac{2n}{n-2s}(\mathbb{R}^n)}\geq C_{n,s}d(f,\mathcal{M}_{HLS}),$$
where $$\mathcal{M}_{HLS}:=\left\{\pm c_{n,s}\left(\frac{\lambda}{1+\lambda^2|x-z|^2}\right)^{\frac{n+2s}{2}}:   z\in \mathbb{R}^n,\  \lambda>0 \right\}$$
and
$$d(f,\mathcal{M}_{HLS})=\inf\limits_{h\in \mathcal{M}_{HLS}}\|f-h\|_{L^{\frac{2n}{n+2s}}(\mathbb{R}^n)}.$$
\end{theorem}

\begin{remark}
In fact, we also obtain that if $\{f_k\}$ is the Palais-Smale for HLS integral equation, then
\begin{equation}\begin{split}
&\lim\limits_{k\rightarrow +\infty}\frac{\left\||f_k|^{-\frac{4s}{n+2s}}f_k-(-\Delta)^{-s}f_k\right\|_{L^\frac{2n}{n-2s}(\mathbb{R}^n)}}{d(f_k,\mathcal{M}_{HLS})}\\
&\ \ \geq \min\{\frac{\Gamma(\frac{n}{2}-s+1)}{\Gamma(\frac{n}{2}+s+1)}\frac{s}{n/2+s+1}
|\mathbb{S}^n|^{-\frac{2s}{n}}\sqrt{\frac{2s}{n+2s}},\ \frac{\Gamma(n/2-s)}{\Gamma(n/2+s)}\frac{2s}{n+2s}|\mathbb{S}^n|^{-\frac{2s}{n}}\}.
\end{split}\end{equation}
To the best of our knowledge, it appears to be the first time to give the explicit lower bound for Palais-Smale sequences of a  functional under non-Hilbert distance.
\end{remark}
\vskip0.3cm
\begin{remark}\label{remark of thm}
We would like to point out that from the proof, we can see the stability inequality still holds when $f$ satisfies $\alpha\mathcal{S}_{s,n}^{\frac{n}{2s}}\leq\int_{\mathbb{R}^n}|f|^{\frac{2n}{n+2s}}dx\leq \beta\mathcal{S}_{s,n}^{\frac{n}{2s}}$ for any $0<\alpha<1<\beta<2$.
\end{remark}

The duality framework for the stability of the Sobolev inequality and the HLS inequality has been systematically developed  by Carlen in \cite{Carlen} and  subsequently by Chen-Lu-Tang in \cite{CLT1, CLT2, CLT3}. However, whether there exists a  duality framework for the stability of the critical points of the Sobolev inequality and the HLS inequality is still unknown. In this paper, we will demonstrate that the stability of critical points for the
HLS inequality can actually imply the stability of critical points of the Sobolev inequality. We will  derive the following
\vskip0.3cm
\begin{theorem}\label{sob-stability}
Let $0<s<n/2$ and $g\in \dot H^s(\mathbb{R}^n)$ with
\begin{equation}\label{condition}
(\frac{1}{2})^{\frac{n-2s}{n}}\mathcal{S}_{s,n}^{\frac{n}{2s}}\leq \|g\|^2_{\dot H^s(\mathbb{R}^n)}\leq (\frac{3}{2})^{\frac{n-2s}{n}}\mathcal{S}_{s,n}^{\frac{n}{2s}}.
\end{equation}
Then there exists a positive constant $G_{n,s}$ such that
$$\left\|(-\Delta)^{s}g-|g|^{\frac{4s}{n-2s}}g\right\|_{\dot H^{-s}(\mathbb{R}^n)}\geq G_{n,s}\inf_{h\in \mathcal{M}_s}\|g-h\|_{\dot H^{s}(\mathbb{R}^n)},$$
where $$ \mathcal{M}_s:=\left\{\pm d_{n,s}\left(\frac{\lambda}{1+\lambda^2|x-z|^2}\right)^{\frac{n-2s}{2}}:   z\in \mathbb{R}^n,\  \lambda>0 \right\}.$$
\end{theorem}
\vskip0.3cm
\begin{remark}
De Nitti and K\"{o}nig's in \cite{NK}  obtain the stability for critical points of the fractional Sobolev inequality for non-negative functions. They also note that
the non-negativity assumption is not necessary when $s\in(0,1]$ since
the Struwe-type profile decompositions (\textbf{Theorem A}) holds for general functions when $s\in(0,1]$. In this paper, we will completely remove the non-negativity  condition for any $s\in (0, \frac{n}{2})$.
\end{remark}
\vskip0.3cm

Let us give a brief overview over the main ideas of the proof of Theorem \ref{stability}. Its basic strategy is to establish the global stability for critical points of the HLS inequality from the local stability and Struwe type of stability for the HLS inequality. We still choose to work on the sphere in order to apply the spherical harmonics technique, which can help us to derive the local stability. The main difficulty is the lack of Hilbert space structure since the metric is induced by $L^{\frac{2n}{n+2s}}$ norm. We develop a weak-decomposition-strong-stability technique in the structure for critical points of HLS inequality and decompose $u$ in $H^{-s}(\mathbb{S}^n)$ as $u=\phi+r$, where $r$ satisfies some orthogonality  condition. Then we try to prove the local stability with $L^{\frac{2n}{n+2s}}$ distance when $\|r\|_{L^{\frac{2n}{n+2s}}}$ is very small. By establishing the comparison theorem between $\inf\limits_{h\in M_{HLS}}\|u-h\|_{\frac{2n}{n+2s}}$ and $\|r\|_{\frac{2n}{n+2s}}$ (see Lemma \ref{infinitesimals}), we derive the local stability for critical points of the HLS inequality when $L^{\frac{2n}{n+2s}}$ distance is very small.
\medskip

The global stability for critical points of the HLS inequality follows from establishing a new Struwe decomposition for general functions rather than nonnegative functions. The main  reason why we can remove the nonnegativity assumption is that we can easily test the HLS integral equation $$|\psi|^{-\frac{4s}{n+2s}}\psi-(-\Delta)^{-s}\psi=0,\ \ x\in \mathbb{R}^n$$ with the positive part and negative part of $\psi$ to derive some energy identities which help us to prove that $\phi$ does not change signs. In contrast, for the fractional Sobolev equation
$$(-\Delta)^s\phi-|\phi|^{\frac{4s}{n-2s}}\phi=0,\ \ x\in \mathbb{R}^n,$$ this method fails since positive part or negative part of $\phi$ may not belong to $\dot{H}^s(\mathbb{R}^n)$ when $s>1$. Hence, we develop a duality framework for critical points of the Sobolev inequality and the HLS inequality to overcome this difficulty and obtain the stability and Struwe decomposition for critical points of the Sobolev inequality without non-negativity  assumption.

\medskip

This paper is organized as follows. Section 2 is devoted to proving the Struwe type decomposition for critical points of the HLS inequality and the Sobolev inequality without nonnegativity   assumption. In Section 3, we will establish the stability for critical point of the HLS and Sobolev inequalities without nonnegativity  assumption. In Section 4, we will provide some lemmas which will be needed in the proof of the main Theorems.

\section{Struwe type decomposition for critical points of HLS inequality and Sobolev inequality without non-negativity assumption }
In this section, we establish the Struwe type decomposition lemma for critical points of HLS inequality and Sobolev inequality without the non-negativity  assumption on functions under consideration. Denote by
$$U_H(z,\lambda)(x)=c_{n,s}\left(\frac{\lambda}{1+\lambda^2|x-z|^2}\right)^{\frac{n+2s}{2}}\in \mathcal{M}_{HLS}^{+},$$
we first establish
\vskip0.3cm
\begin{lemma}\label{deHLS}
For $0<s<\frac{n}{2}$ and $\{f_k\}_k\in L^{\frac{2n}{n+2s}}(\mathbb{R}^n)$ satisfies
$$\frac{1}{2}\mathcal{S}_{s,n}^{\frac{n}{2s}}\leq \int_{\mathbb{R}^n}|f_k|^{\frac{2n}{n+2s}}dx\leq \frac{3}{2}\mathcal{S}_{s,n}^{\frac{n}{2s}}$$
and
$$\lim\limits_{k\rightarrow +\infty}\left\||f_k|^{-\frac{4s}{n+2s}}f_k-(-\Delta)^{-s}f_k\right\|_{L^{\frac{2n}{n-2s}}(\mathbb{R}^n)}=0,$$
Then there exists a subsequence of $\{f_k\}_k$ (still denoted by $\{f_k\}_k$) and $\{z_k\}_k\subset \mathbb{R}^n$, $\{\lambda_k\}_k\subset \mathbb{R}^{+}$ such that
$$\lim\limits_{k\rightarrow +\infty}\Big(|f_k|^{-\frac{4s}{n+2s}}f_k-\beta U_{H}^{-\frac{4s}{n+2s}}(z_k,\lambda_k)U_{H}(z_k,\lambda_k), f_k-\beta U_{H}(z_k,\lambda_k)\Big)=0,$$
where $\beta$ is equal to $+1$ or $-1$. Furthermore, we can derive that
$$\lim_{k\rightarrow +\infty}\int_{\mathbb{R}^n}|f_k-\beta U_{H}|^{\frac{2n}{n+2s}}dx=0.$$
\end{lemma}
\vskip0.3cm
\begin{remark}\label{remark}
As we pointed out in the remark of Theorem \ref{stability}, the constant $\frac{1}{2}$ and $\frac{3}{2}$ here can be replaced by any $0<\alpha<1$ and $1<\beta<2$ respectively.
\end{remark}
In order to prove the above Struwe type decomposition lemma, we need the following profile decomposition result which can be found in \cite{BCK}
\begin{lemma}\label{rohls}
Assume $0<s<\frac{n}{2}$ and that $f_k$ is uniformly bounded in $L^{\frac{2n}{n+2s}}(\mathbb{R}^n)$. Then there exists a subsequence of $\{f_k\}_k$ (still denoted by $\{f_k\}_k$) such that $f_k$ can be written as
$$f_k=\sum_{j=1}^{l}\big(h_k^{j}\big)^{-\frac{n+2s}{2}}\psi_j(\frac{\cdot-x_k^j}{h_k^j})+r_k^{l},$$
where $\psi_j\in L^{\frac{2n}{n+2s}}(\mathbb{R}^n)$, $h_k^{j}\in \mathbb{R}^{+}$ and $x_{k}^{j}\in \mathbb{R}^n$ satisfying
$$\lim\limits_{k\rightarrow +\infty}\big(|\log \frac{h_k^{i}}{h_k^{j}}| +\frac{|x_k^{i}-x_k^{j}|}{h_k^{i}}\big)= +\infty,\  i\neq j.$$
and the remainder term $r_k^{l}$ satisfies $\lim\limits_{l\rightarrow +\infty}\lim\limits_{k\rightarrow +\infty}\|r_k^{l}\|_{H^{-s}(\mathbb{R}^n)}=0$. Moreover, $$\|f_k\|_{L^{\frac{2n}{n+2s}}(\mathbb{R}^n)}^{\frac{2n}{n+2s}}=\sum_{j=1}^{l}\|\psi_j\|_{L^{\frac{2n}{n+2s}}(\mathbb{R}^n)}^{\frac{2n}{n+2s}}+\|r_{k}^{l}\|_{L^{\frac{2n}{n+2s}}(\mathbb{R}^n)}^{\frac{2n}{n+2s}}+o_l(1)$$
as $l\rightarrow +\infty$.
\end{lemma}

Now, let us prove the Struwe type decomposition lemma for critical points of the HLS inequality.
\begin{proof}
By Lemma \ref{rohls},  there exist $\psi_j\in L^{\frac{2n}{n+2s}}(\mathbb{R}^n)$ and $h_k^{j}\in \mathbb{R}^{+}$ and $x_{k}^{j}\in \mathbb{R}^n$ satisfying
\begin{equation}\begin{split}\label{split}
\lim\limits_{k\rightarrow +\infty}\big(|\log \frac{h_k^{i}}{h_k^{j}}|+\frac{|x_k^{i}-x_k^{j}|}{h_k^{i}}\big)= +\infty,\ \ \forall i\neq j.
\end{split}\end{equation}
such that $f_k$ can be written as
$$f_k=\sum_{j=1}^{l}\big(h_k^{j}\big)^{-\frac{n+2s}{2}}\psi_j(\frac{\cdot-x_k^j}{h_k^j})+r_k^{l}$$
and
$$\|f_k\|_{L^{\frac{2n}{n+2s}}(\mathbb{R}^n)}^2=\sum_{j=1}^{l}\|\psi_j\|_{L^{\frac{2n}{n+2s}}(\mathbb{R}^n)}^2+\|r_{k}^{l}\|_{L^{\frac{2n}{n+2s}}(\mathbb{R}^n)}^2+o_l(1).$$
Since $\|f_k\|_{L^{\frac{2n}{n+2s}}(\mathbb{R}^n)}^2$ is bounded from below, then not all $\psi_j$ are zero. So, we may assume that $\psi_1\neq 0$. Since $f_k$ is a Palais-Smale sequence for the HLS integral equation:
\begin{equation}\label{ps}\lim\limits_{k\rightarrow +\infty}\left\||f_k|^{-\frac{4s}{n+2s}}f_k-(-\Delta)^{-s}f_k\right\|_{L^{\frac{2n}{n-2s}}(\mathbb{R}^n)}=0.
\end{equation}
We first claim that $\psi_1$ must satisfy the equation $$|\psi_1|^{-\frac{4s}{n+2s}}\psi_1-(-\Delta)^{-s}\psi_1=0$$
in the weak sense.
Indeed, fix a test function $\phi$, then for $\phi_k^1=(h_k^1)^{-\frac{n+2s}{2}}\phi(\frac{.-x_k^1}{h_k^1})$, through \eqref{ps}, we get
\begin{equation}\begin{split}
o_k(1)&=\big(|f_k|^{-\frac{4s}{n+2s}}f_k-(-\Delta)^{-s}f_k, \phi_k^1\big).\\
\end{split}\end{equation}
If we can prove that
\begin{equation}\label{key-inequality}\begin{split}
&\lim\limits_{k\rightarrow +\infty}\big(|f_k|^{-\frac{4s}{n+2s}}f_k-(-\Delta)^{-s}f_k, \phi_k^1\big)\\
&\ \ =\lim\limits_{k\rightarrow +\infty} \Big(|(h_k^1)^{-\frac{n+2s}{2}}\psi_1(\frac{.-x_k^1}{h_k^1})|^{-\frac{4s}{n+2s}}(h_k^1)^{-\frac{n+2s}{2}}\psi_1(\frac{.-x_k^1}{h_k^1})-(-\Delta)^{-s}\big((h_k^1)^{-\frac{n+2s}{2}}\psi_1(\frac{.-x_k^1}{h_k^1})\big),\phi_k^1\Big)\\
&\ \ = \Big(|\psi_1|^{-\frac{4s}{n+2s}}\psi_1-(-\Delta)^{-s}\psi_1, \phi\Big),
\end{split}\end{equation}
we can directly get $\psi_1$ satisfies equation $|\psi_1|^{-\frac{4s}{n+2s}}\psi_1-(-\Delta)^{-s}\psi_1=0$ in the weak sense. Let $F(t)=|t|^{\frac{-4s}{n+2s}}t$, using the elementary inequality ( The proof is presented in the Appendix) $$F(a+b)-C_{n,s}|b|^{\frac{n-2s}{n+2s}} \leq F(a)\leq F(a+b)+C_{n,s}|b|^{\frac{n-2s}{n+2s}},$$
we derive
\begin{equation*}\begin{split}
&\big(|f_k|^{-\frac{4s}{n+2s}}f_k-(-\Delta)^{-s}f_k, \phi_k^1\big)\\
&\ \ =\big(|\sum_{j=1}^{l}\big(h_k^{j}\big)^{-\frac{n+2s}{2}}\psi_j(\frac{\cdot-x_k^j}{h_k^j})+r_k^{l}|^{-\frac{4s}{n+2s}}(\sum_{j=1}^{l}\big(h_k^{j}\big)^{-\frac{n+2s}{2}}\psi_j(\frac{\cdot-x_k^j}{h_k^j})+r_k^{l}), \phi_k^1\big)\\
&\ \ \ \ -\big((-\Delta)^{-s}(\sum_{j=1}^{l}\big(h_k^{j}\big)^{-\frac{n+2s}{2}}\psi_j(\frac{\cdot-x_k^j}{h_k^j})+r_k^{l}), \phi_k^1\big)\\
&\ \ \leq \big(|\big(h_k^{1}\big)^{-\frac{n+2s}{2}}\psi_1(\frac{\cdot-x_k^1}{h_k^1})|^{-\frac{4s}{n+2s}}(\big(h_k^{1}\big)^{-\frac{n+2s}{2}}\psi_1(\frac{\cdot-x_k^j}{h_k^1})), \phi_k^1\big) \\
&\ \ \ \ +C_{n,s}\big(|\sum_{j\neq 1}\big(h_k^{j}\big)^{-\frac{n+2s}{2}}\psi_j(\frac{\cdot-x_k^j}{h_k^j})+r_k^{l}|^{\frac{n-2s}{n+2s}}, |\phi_k^1|\big)-\big((-\Delta)^{-s}(\sum_{j=1}^{l}\big(h_k^{j}\big)^{-\frac{n+2s}{2}}\psi_j(\frac{\cdot-x_k^j}{h_k^j})+r_k^{l}), \phi_k^1\big)
\end{split}\end{equation*}
and

\begin{equation*}\begin{split}
&\big(|f_k|^{-\frac{4s}{n+2s}}f_k-(-\Delta)^{-s}f_k, \phi_k^1\big)\\
&\ \ \geq \big(|\big(h_k^{1}\big)^{-\frac{n+2s}{2}}\psi_1(\frac{\cdot-x_k^1}{h_k^1})|^{-\frac{4s}{n+2s}}(\big(h_k^{1}\big)^{-\frac{n+2s}{2}}\psi_1(\frac{\cdot-x_k^j}{h_k^1})), \phi_k^1\big) \\
&\ \ \ \ -C_{n,s}\big(|\sum_{j\neq 1}\big(h_k^{j}\big)^{-\frac{n+2s}{2}}\psi_j(\frac{\cdot-x_k^j}{h_k^j})+r_k^{l}|^{\frac{n-2s}{n+2s}}, |\phi_k^1|\big)-\big((-\Delta)^{-s}(\sum_{j=1}^{l}\big(h_k^{j}\big)^{-\frac{n+2s}{2}}\psi_j(\frac{\cdot-x_k^j}{h_k^j})+r_k^{l}), \phi_k^1\big)
\end{split}\end{equation*}

Hence, in order to prove inequality \eqref{key-inequality}, it suffices to prove that for $j\neq 1$,
\begin{equation}\label{str1}
\lim\limits_{k\rightarrow +\infty}\int_{\mathbb{R}^n}|(h_k^j)^{-\frac{n+2s}{2}}\psi_j(\frac{.-x_k^j}{h_k^j})|^{-\frac{4s}{n+2s}}|(h_k^j)^{-\frac{n+2s}{2}}\psi_j(\frac{.-x_k^j}{h_k^j})\phi_k^1|dx=0
\end{equation}
\begin{equation}\label{str2}
\lim\limits_{k\rightarrow +\infty}\int_{\mathbb{R}^n}(-\Delta)^{-s}\big((h_k^j)^{-\frac{n+2s}{2}}\psi_j(\frac{.-x_k^j}{h_k^j})\big)\phi_k^1dx=0.
\end{equation}
According to $$\lim\limits_{k\rightarrow +\infty} \big(|\log \frac{h_k^{i}}{h_k^{j}}| +\frac{|x_k^{i}-x_k^{j}|}{h_k^{i}}\big)= +\infty,\  i\neq j,$$
we have $\lim\limits_{k\rightarrow +\infty}|\log \frac{h_k^1}{h_k^j}|=+\infty$ or $\lim\limits_{k\rightarrow +\infty}\frac{|x_k^{1}-x_k^{j}|}{h_k^{1}}= +\infty$ for $j\neq 1$.
We may first assume that $\lim\limits_{k\rightarrow +\infty}|\log \frac{h_k^1}{h_k^j}|=+\infty$, we distinguish between the case $\lim\limits_{k\rightarrow +\infty}\frac{h_k^j}{h_k^1}=0$ and the case $\lim\limits_{k\rightarrow +\infty}\frac{h_k^1}{h_k^j}=0$.
\vskip0.1cm

Case 1: $\lim\limits_{k\rightarrow +\infty}\frac{h_k^j}{h_k^1}=0$. Careful calculation gives that
\begin{equation*}\begin{split}
&\int_{\mathbb{R}^n}|(h_k^j)^{-\frac{n+2s}{2}}\psi_j(\frac{.-x_k^j}{h_k^j})|^{\frac{n-2s}{n+2s}}|(h_k^1)^{-\frac{n+2s}{2}}\phi(\frac{.-x_k^1}{h_k^1})|dx\\
&\ \ =\int_{\mathbb{R}^n} (\frac{h_k^j}{h_k^1})^{\frac{n+2s}{2}}|\psi_j(y)|^{\frac{n-2s}{n+2s}}\phi(\frac{x_k^j-x_k^1}{h_k^1}+y \frac{h_k^j}{h_k^1})dy\\
& \ \ \leq \Big(\int_{|y|>R}|\psi_j(y)|^{\frac{2n}{n+2s}}dy\Big)^{\frac{n-2s}{2n}}\Big(\int_{\mathbb{R}^n}(\frac{h_k^j}{h_k^1})^n |\phi(\frac{x_k^j-x_k^1}{h_k^1}+y \frac{h_k^j}{h_k^1})|^{\frac{2n}{n+2s}}dy\Big)^{\frac{n+2s}{2n}}\\
&\ \ \ \ +\Big(\int_{\mathbb{R}^n}|\psi_j(y)|^{\frac{2n}{n+2s}}dy\Big)^{\frac{n-2s}{2n}}\Big(\int_{|y|<R}(\frac{h_k^j}{h_k^1})^n |\phi(\frac{x_k^j-x_k^1}{h_k^1}+y \frac{h_k^j}{h_k^1})|^{\frac{2n}{n+2s}}dy\Big)^{\frac{n+2s}{2n}}\\
&\ \ = \Big(\int_{|y|>R}|\psi_j(y)|^{\frac{2n}{n+2s}}dy\Big)^{\frac{n-2s}{2n}}\|\phi\|_{L^{\frac{2n}{n+2s}}(\mathbb{R}^n)}^{\frac{n-2s}{n+2s}}\\
&\ \ \ \ +\|\psi_j\|^{\frac{n-2s}{n+2s}}_{L^{\frac{2n}{n+2s}}(\mathbb{R}^n)}\Big(\int_{|y|<R}(\frac{h_k^j}{h_k^1})^n |\phi(\frac{x_k^j-x_k^1}{h_k^1}+y \frac{h_k^j}{h_k^1})|^{\frac{2n}{n+2s}}dy\Big)^{\frac{n+2s}{2n}}.\\
\end{split}\end{equation*}
For $\int_{|y|>R}|\psi_j(y)|^{\frac{2n}{n+2s}}dy$, absolute continuity and $L^{\frac{2n}{n+2s}}(\mathbb{R}^n)$ integrability of $\psi_j$ leads to
$$\lim\limits_{R\rightarrow +\infty}\lim\limits_{k\rightarrow +\infty}\int_{|y|>R}|\psi_j(y)|^{\frac{2n}{n+2s}}dy=0.$$ For
\begin{equation}\label{inter}
\int_{|y|<R}(\frac{h_k^j}{h_k^1})^n |\phi(\frac{x_k^j-x_k^1}{h_k^1}+y \frac{h_k^j}{h_k^1})|^{\frac{2n}{n+2s}}dy,
\end{equation}
if $\lim\limits_{k\rightarrow +\infty}\frac{|x_k^{1}-x_k^{j}|}{h_k^{1}}= +\infty$, then the change of variable gives that
$$\lim\limits_{k\rightarrow +\infty}\int_{|y|<R}(\frac{h_k^j}{h_k^1})^n |\phi(\frac{x_k^j-x_k^1}{h_k^1}+y \frac{h_k^j}{h_k^1})|^{\frac{2n}{n+2s}}dy\leq
\lim\limits_{k\rightarrow +\infty}\int_{|z|\geq \frac{x_k^j-x_k^1}{h_k^1}-\frac{h_k^j}{h_k^1}R}|\phi(z)|^{\frac{2n}{n+2s}}dz=0;$$
if $\frac{|x_k^{1}-x_k^{j}|}{h_k^{1}}$ is bounded, the change of variable gives that
$$\lim\limits_{k\rightarrow +\infty}\int_{|y|<R}(\frac{h_k^j}{h_k^1})^n |\phi(\frac{x_k^j-x_k^1}{h_k^1}+y \frac{h_k^j}{h_k^1})|^{\frac{2n}{n+2s}}dy\leq
\lim\limits_{k\rightarrow +\infty}\int_{|z-\frac{x_k^j-x_k^1}{h_k^1}|\leq\frac{h_k^j}{h_k^1}R}|\phi(z)|^{\frac{2n}{n+2s}}dz=0.$$
Hence, we derive that
\begin{equation*}\begin{split}
\lim\limits_{k\rightarrow +\infty}\int_{\mathbb{R}^n}|(h_k^j)^{-\frac{n+2s}{2}}\psi_j(\frac{.-x_k^j}{h_k^j})|^{\frac{n-2s}{n+2s}}|(h_k^1)^{-\frac{n+2s}{2}}\phi(\frac{.-x_k^1}{h_k^1})|dx=0.
\end{split}\end{equation*}

Case 2: $\lim\limits_{k\rightarrow +\infty}\frac{h_k^1}{h_k^j}=0$. Careful calculation gives that
\begin{equation*}\begin{split}
&\int_{\mathbb{R}^n}|(h_k^j)^{-\frac{n+2s}{2}}\psi_j(\frac{.-x_k^j}{h_k^j})|^{\frac{n-2s}{n+2s}}|(h_k^1)^{-\frac{n+2s}{2}}\phi(\frac{.-x_k^1}{h_k^1})|dx\\
&\ \ =\int_{\mathbb{R}^n} (\frac{h_k^1}{h_k^j})^{\frac{n-2s}{2}}|\psi_j(\frac{x_k^j-x_k^1}{h_k^j}+y \frac{h_k^1}{h_k^j})|^{\frac{n-2s}{n+2s}}\phi(y)dy\\
& \ \ \leq \Big(\int_{|y|>R}|\phi(y)|^{\frac{2n}{n+2s}}dy\Big)^{\frac{n+2s}{2n}}\Big(\int_{\mathbb{R}^n}(\frac{h_k^1}{h_k^j})^n |\psi_j(\frac{x_k^j-x_k^1}{h_k^j}+y \frac{h_k^1}{h_k^j})|^{\frac{2n}{n+2s}}dy\Big)^{\frac{n-2s}{2n}}\\
&\ \ \ \ +\Big(\int_{\mathbb{R}^n}|\phi(y)|^{\frac{2n}{n+2s}}dy\Big)^{\frac{n+2s}{2n}}\Big(\int_{|y|<R}(\frac{h_k^1}{h_k^j})^n |\psi_j(\frac{x_k^j-x_k^1}{h_k^j}+y \frac{h_k^1}{h_k^j})|^{\frac{2n}{n+2s}}dy\Big)^{\frac{n-2s}{2n}}\\
&\ \ = \Big(\int_{|y|>R}|\phi(y)|^{\frac{2n}{n+2s}}dy\Big)^{\frac{n+2s}{2n}}\|\psi_j\|_{L^{\frac{2n}{n+2s}}(\mathbb{R}^n)}^{\frac{n-2s}{n+2s}}\\
&\ \ \ \ +\|\phi\|^{\frac{2n}{n+2s}}_{L^{\frac{2n}{n+2s}}(\mathbb{R}^n)}\Big(\int_{|y|<R}(\frac{h_k^1}{h_k^j})^n |\psi_j(\frac{x_k^j-x_k^1}{h_k^j}+y \frac{h_k^1}{h_k^j})|^{\frac{2n}{n+2s}}dy\Big)^{\frac{n-2s}{2n}}.\\
\end{split}\end{equation*}
For $\int_{|y|>R}|\phi(y)|^{\frac{2n}{n+2s}}dy$, absolute continuity and $L^{\frac{2n}{n+2s}}(\mathbb{R}^n)$ integrability of $\phi$ leads to
$$\lim\limits_{R\rightarrow +\infty}\lim\limits_{k\rightarrow +\infty}\int_{|y|>R}|\phi(y)|^{\frac{2n}{n+2s}}dy=0.$$ For
$$\int_{|y|<R}(\frac{h_k^1}{h_k^j})^n |\psi_j(\frac{x_k^j-x_k^1}{h_k^j}+y \frac{h_k^1}{h_k^j})|^{\frac{2n}{n+2s}}dy,$$
if $\lim\limits_{k\rightarrow +\infty}\frac{|x_k^{1}-x_k^{j}|}{h_k^{j}}= +\infty$, then the change of variable gives that
$$\lim\limits_{k\rightarrow +\infty}\int_{|y|<R}(\frac{h_k^1}{h_k^j})^n |\psi_j(\frac{x_k^j-x_k^1}{h_k^j}+y \frac{h_k^1}{h_k^j})|^{\frac{2n}{n+2s}}dy\leq
\lim\limits_{k\rightarrow +\infty}\int_{|z|\geq \frac{x_k^j-x_k^1}{h_k^j}-\frac{h_k^1}{h_k^j}R}|\psi_j(z)|^{\frac{2n}{n+2s}}dz=0;$$
if $\frac{|x_k^{1}-x_k^{j}|}{h_k^{1}}$ is bounded, the change of variable gives that
$$\lim\limits_{k\rightarrow +\infty}\int_{|y|<R}(\frac{h_k^j}{h_k^j})^n |\psi_j(\frac{x_k^j-x_k^1}{h_k^j}+y \frac{h_k^1}{h_k^j})|^{\frac{2n}{n+2s}}dy\leq
\lim\limits_{k\rightarrow +\infty}\int_{|z-\frac{x_k^j-x_k^1}{h_k^j}|\leq\frac{h_k^1}{h_k^j}R}|\psi_j(z)|^{\frac{2n}{n+2s}}dz=0.$$
Then
\begin{equation*}\begin{split}
\lim\limits_{k\rightarrow +\infty}\int_{\mathbb{R}^n}|(h_k^j)^{-\frac{n+2s}{2}}\psi_j(\frac{.-x_k^j}{h_k^j})|^{\frac{n-2s}{n+2s}}|(h_k^1)^{-\frac{n+2s}{2}}\phi(\frac{.-x_k^1}{h_k^1})|dx=0.
\end{split}\end{equation*}

Now, we assume that $\lim\limits_{k\rightarrow +\infty}|\log \frac{h_k^1}{h_k^j}|=+\infty$ does not hold, then $\lim\limits_{k\rightarrow +\infty}\frac{|x_k^{1}-x_k^{j}|}{h_k^{1}}= +\infty$ and $\frac{h_k^j}{h_k^1}$ is bounded for $j\neq 1$. Similar calculation gives that \begin{equation*}\begin{split}
&\int_{\mathbb{R}^n}|(h_k^j)^{-\frac{n+2s}{2}}\psi_j(\frac{.-x_k^j}{h_k^j})|^{\frac{n-2s}{n+2s}}|(h_k^1)^{-\frac{n+2s}{2}}\phi(\frac{.-x_k^1}{h_k^1})|dx\\
&\ \ =\int_{\mathbb{R}^n} (\frac{h_k^j}{h_k^1})^{\frac{n+2s}{2}}|\psi_j(y)|^{\frac{n-2s}{n+2s}}\phi(\frac{x_k^j-x_k^1}{h_k^1}+y \frac{h_k^j}{h_k^1})dy\\
& \ \ \leq \Big(\int_{|y|>R}|\psi_j(y)|^{\frac{2n}{n+2s}}dy\Big)^{\frac{n-2s}{2n}}\Big(\int_{\mathbb{R}^n}(\frac{h_k^j}{h_k^1})^n |\phi(\frac{x_k^j-x_k^1}{h_k^1}+y \frac{h_k^j}{h_k^1})|^{\frac{2n}{n+2s}}dy\Big)^{\frac{n+2s}{2n}}\\
&\ \ \ \ +\Big(\int_{\mathbb{R}^n}|\psi_j(y)|^{\frac{2n}{n+2s}}dy\Big)^{\frac{n-2s}{2n}}\Big(\int_{|y|<R}(\frac{h_k^j}{h_k^1})^n |\phi(\frac{x_k^j-x_k^1}{h_k^1}+y \frac{h_k^j}{h_k^1})|^{\frac{2n}{n+2s}}dy\Big)^{\frac{n+2s}{2n}}\\
&\ \ = \Big(\int_{|y|>R}|\psi_j(y)|^{\frac{2n}{n+2s}}dy\Big)^{\frac{n-2s}{2n}}\|\phi\|_{L^{\frac{2n}{n+2s}}(\mathbb{R}^n)}^{\frac{n-2s}{n+2s}}\\
&\ \ \ \ +\|\psi_j\|^{\frac{n-2s}{n+2s}}_{L^{\frac{2n}{n+2s}}(\mathbb{R}^n)}\Big(\int_{|y|<R}(\frac{h_k^j}{h_k^1})^n |\phi(\frac{x_k^j-x_k^1}{h_k^1}+y \frac{h_k^j}{h_k^1})|^{\frac{2n}{n+2s}}dy\Big)^{\frac{n+2s}{2n}}.\\
\end{split}\end{equation*}
For $\int_{|y|>R}|\psi_j(y)|^{\frac{2n}{n+2s}}dy$, absolute continuity and $L^{\frac{2n}{n+2s}}(\mathbb{R}^n)$ integrability of $\psi_j$ leads to
$$\lim\limits_{R\rightarrow +\infty}\lim\limits_{k\rightarrow +\infty}\int_{|y|>R}|\psi_j(y)|^{\frac{2n}{n+2s}}dy=0.$$ For
\begin{equation}\label{inter}
\int_{|y|<R}(\frac{h_k^j}{h_k^1})^n |\phi(\frac{x_k^j-x_k^1}{h_k^1}+y \frac{h_k^j}{h_k^1})|^{\frac{2n}{n+2s}}dy,
\end{equation}
the change of variable gives that
$$\lim\limits_{k\rightarrow +\infty}\int_{|y|<R}(\frac{h_k^j}{h_k^1})^n |\phi(\frac{x_k^j-x_k^1}{h_k^1}+y \frac{h_k^j}{h_k^1})|^{\frac{2n}{n+2s}}dy\leq
\lim\limits_{k\rightarrow +\infty}\int_{|z|\geq \frac{x_k^j-x_k^1}{h_k^1}-\frac{h_k^j}{h_k^1}R}|\phi(z)|^{\frac{2n}{n+2s}}dz=0$$
since $\frac{h_k^j}{h_k^1}$ is bounded. Hence, we also derive that
\begin{equation*}\begin{split}
\lim\limits_{k\rightarrow +\infty}\int_{\mathbb{R}^n}|(h_k^j)^{-\frac{n+2s}{2}}\psi_j(\frac{.-x_k^j}{h_k^j})|^{\frac{n-2s}{n+2s}}|(h_k^1)^{-\frac{n+2s}{2}}\phi(\frac{.-x_k^1}{h_k^1})|dx=0.
\end{split}\end{equation*}
In summary, we prove the equality \eqref{str1}. Now, we start to prove that $$\lim\limits_{k\rightarrow +\infty}\Big((-\Delta)^{-s}\big((h_k^j)^{-\frac{n+2s}{2}}\psi_j(\frac{.-x_k^j}{h_k^j})\big),\phi_k^1\Big)=0.$$
Careful calculation gives that
\begin{equation*}\begin{split}
&\Big((-\Delta)^{-s}\big((h_k^j)^{-\frac{n+2s}{2}}\psi_j(\frac{.-x_k^j}{h_k^j})\big),\phi_k^1\Big)\\
&\ \ =\frac{\Gamma(n/2-s)}{2^{2s}\pi^{n/2}\Gamma(s)}\int_{\mathbb{R}^n}\int_{\mathbb{R}^n}\big((h_k^j)^{-\frac{n+2s}{2}}\psi_j(\frac{x'-x_k^j}{h_k^j})\big)\frac{1}{|x-x'|^{n-2s}}(h_k^1)^{-\frac{n+2s}{2}}\phi(\frac{x-x_k^1}{h_k^1})dx'dx\\
&=\frac{\Gamma(n/2-s)}{2^{2s}\pi^{n/2}\Gamma(s)}\big(\frac{h_k^j}{h_k^1}\big)^{\frac{n+2s}{2}}\int_{\mathbb{R}^n}\int_{\mathbb{R}^n}\psi_j(y')\frac{1}{|y'-z'|^{n-2s}}\phi(\frac{h_k^j}{h_k^1}z'+\frac{x_k^j-x_k^1}{h_k^1})dy'dz'\\
&=\big(\frac{h_k^j}{h_k^1}\big)^{\frac{n+2s}{2}}\int_{\mathbb{R}^n}(-\Delta)^{-\frac{s}{2}}(\psi_j)(-\Delta)^{-\frac{s}{2}}(\phi(\frac{h_k^j}{h_k^1}.+\frac{x_k^j-x_k^1}{h_k^1}))dx\\
&\leq \Big(\int_{|x|>R}|(-\Delta)^{-\frac{s}{2}}\psi_j(x)|^2dx\Big)^{\frac{1}{2}}\int_{\mathbb{R}^n} \Big( \big(\frac{h_k^j}{h_k^1}\big)^{n}|(-\Delta)^{-\frac{s}{2}}\phi(\frac{h_k^j}{h_k^1}x+\frac{x_k^j-x_k^1}{h_k^1})|^2dx\Big)^{\frac{1}{2}}\\
&\ \ +\Big(\int_{\mathbb{R}^n}|(-\Delta)^{-\frac{s}{2}}\psi_j(x)|^2dx\Big)^{\frac{1}{2}}\int_{|x|<R} \Big( \big(\frac{h_k^j}{h_k^1}\big)^{n}|(-\Delta)^{-\frac{s}{2}}\phi(\frac{h_k^j}{h_k^1}x+\frac{x_k^j-x_k^1}{h_k^1})|^2dx\Big)^{\frac{1}{2}}\\
&\leq \Big(\int_{|x|>R}|(-\Delta)^{-\frac{s}{2}}\psi_j(x)|^2dx\Big)^{\frac{1}{2}} \|(-\Delta)^{-\frac{s}{2}}\phi\|_{L^2(\mathbb{R}^n)}\\
&\ \ +\|(-\Delta)^{-\frac{s}{2}}\psi_j\|_{L^2(\mathbb{R}^n)}\int_{|x|<R} \Big( \big(\frac{h_k^j}{h_k^1}\big)^{n}|(-\Delta)^{-\frac{s}{2}}\phi(\frac{h_k^j}{h_k^1}x+\frac{x_k^j-x_k^1}{h_k^1})|^2dx\Big)^{\frac{1}{2}}\\
\end{split}\end{equation*}
and \begin{equation*}\begin{split}
&\Big((-\Delta)^{-s}\big((h_k^j)^{-\frac{n+2s}{2}}\psi_j(\frac{.-x_k^j}{h_k^j})\big),\phi_k^1\Big)\\
&\leq \Big(\int_{|x|>R}|(-\Delta)^{-\frac{s}{2}}\phi(x)|^2dx\Big)^{\frac{1}{2}}\int_{\mathbb{R}^n} \Big( \big(\frac{h_k^1}{h_k^j}\big)^{n}|(-\Delta)^{-\frac{s}{2}}\psi_j(\frac{h_k^1}{h_k^j}x+\frac{x_k^j-x_k^1}{h_k^j})|^2dx\Big)^{\frac{1}{2}}\\
&\ \ +\Big(\int_{\mathbb{R}^n}|(-\Delta)^{-\frac{s}{2}}\phi(x)|^2dx\Big)^{\frac{1}{2}}\int_{|x|<R} \Big( \big(\frac{h_k^1}{h_k^j}\big)^{n}|(-\Delta)^{-\frac{s}{2}}\psi_j(\frac{h_k^1}{h_k^j}x+\frac{x_k^j-x_k^1}{h_k^j})|^2dx\Big)^{\frac{1}{2}}\\
&\leq \Big(\int_{|x|>R}|(-\Delta)^{-\frac{s}{2}}\phi(x)|^2dx\Big)^{\frac{1}{2}} \|(-\Delta)^{-\frac{s}{2}}\psi_j\|_{L^2(\mathbb{R}^n)}\\
&\ \ +\|(-\Delta)^{-\frac{s}{2}}\phi\|_{L^2(\mathbb{R}^n)}\int_{|x|<R} \Big( \big(\frac{h_k^1}{h_k^j}\big)^{n}|(-\Delta)^{-\frac{s}{2}}\psi_j(\frac{h_k^1}{h_k^j}x+\frac{x_k^j-x_k^1}{h_k^j})|^2dx\Big)^{\frac{1}{2}}\\
\end{split}\end{equation*}
Similar to the proof of equality \eqref{str1}, we can conclude that $$\lim\limits_{k\rightarrow +\infty}\Big((-\Delta)^{-s}\big((h_k^j)^{-\frac{n+2s}{2}}\psi_j(\frac{.-x_k^j}{h_k^j})\big),\phi_k^1\Big)=0.$$
All in all, we conclude that $\psi_1$ satisfies equation $|\psi_1|^{-\frac{4s}{n+2s}}\psi_1-(-\Delta)^{-s}\psi_1=0$. Similarly, we can also show that $\psi_i$ satisfies equation $|\psi_j|^{-\frac{4s}{n+2s}}\psi_j-(-\Delta)^{-s}\psi_j=0$. Now, we claim that $\psi_j=0$ for $j\neq 1$. Indeed,
$$\mathcal{S}_{s,n}\leq \frac{\|\psi_j\|_{L^{\frac{2n}{n+2s}}(\mathbb{R}^n)}^2}{\|(-\Delta)^{-\frac{s}{2}}\psi_j\|_{L^2(\mathbb{R}^n)}^2}= \frac{\|\psi_j\|_{L^{\frac{2n}{n+2s}}(\mathbb{R}^n)}^2}{\|\psi_j\|_{L^{\frac{2n}{n+2s}}(\mathbb{R}^n)}^{\frac{2n}{n+2s}}},$$
this leads to $$\|\psi_j\|_{L^{\frac{2n}{n+2s}}(\mathbb{R}^n)}^{\frac{2n}{n+2s}}\geq \mathcal{S}_{s,n}^{\frac{n}{2s}}$$
Hence
\begin{equation*}\begin{split}
\|f_k\|_{L^{\frac{2n}{n+2s}}(\mathbb{R}^n)}^{\frac{2n}{n+2s}}&=\sum_{j=1}^{l}\|\psi_j\|_{L^{\frac{2n}{n+2s}}(\mathbb{R}^n)}^{\frac{2n}{n+2s}}+\|r_{k}^{l}\|_{L^{\frac{2n}{n+2s}}(\mathbb{R}^n)}^{\frac{2n}{n+2s}}+o_l(1)\\
&\geq 2\mathcal{S}_{s,n}^{\frac{n}{2s}}+\|r_{k}^{l}\|_{L^{\frac{2n}{n+2s}}(\mathbb{R}^n)}^{\frac{2n}{n+2s}}+o_l(1)
\end{split}\end{equation*}
if there exists $j\neq 1$ such that $\psi_j\neq 0$. This is a contradiction with the assumption of $f_k$
 $$\frac{1}{2}\mathcal{S}_{s,n}^{\frac{n}{2s}}\leq \int_{\mathbb{R}^n}|f_k|^{\frac{2n}{n+2s}}dx\leq \frac{3}{2}\mathcal{S}_{s,n}^{\frac{n}{2s}}.$$
Hence we conclude that $\psi_j=0$ for $j\neq 1$. Now, we prove that $\psi_1\in \mathcal{M}_{HLS}$. Since we have already known that the positive solution of equation \begin{equation}\label{el equ}
|\psi_1|^{-\frac{4s}{n+2s}}\psi_1-(-\Delta)^{-s}\psi_1=0
\end{equation}
belongs to $\mathcal{M}_{HLS}^{+}$, we only need to prove that $\psi_1$ does not change signs. Testing the above equation (\ref{el equ}) with $\psi_1^{+}$ and $\psi_1^{-}$, we derive that
\begin{equation*}
\|\psi_1^{+}\|_{L^{\frac{2n}{n+2s}}(\mathbb{R}^n)}^{\frac{2n}{n+2s}}=\|(-\Delta)^{-\frac{s}{2}}\psi_1^{+}\|_{L^2(\mathbb{R}^n)}^2-\int_{\mathbb{R}^n}(-\Delta)^{-s}\psi_1^{-}\cdot\psi_1^{+}dx
\end{equation*}
and
\begin{equation*}
\|\psi_1^{-}\|_{L^{\frac{2n}{n+2s}}(\mathbb{R}^n)}^{\frac{2n}{n+2s}}=\|(-\Delta)^{-\frac{s}{2}}\psi_1^{-}\|_{L^2(\mathbb{R}^n)}^2-
\int_{\mathbb{R}^n}(-\Delta)^{-s}\psi_1^{+}\cdot\psi_1^{-}dx.
\end{equation*}
As a consequence of HLS inequality, we have
$$\mathcal{S}_{s,n}^{\frac{n}{n+2s}}\leq \frac{\|\psi_1^{+}\|_{L^{\frac{2n}{n+2s}}(\mathbb{R}^n)}^{\frac{2n}{n+2s}}}{\|(-\Delta)^{-\frac{s}{2}}\psi_1^{+}\|_{L^2(\mathbb{R}^n)}^{\frac{2n}{n+2s}}}=\frac{\|(-\Delta)^{-\frac{s}{2}}\psi_1^{+}\|_{L^2(\mathbb{R}^n)}^2-\int_{\mathbb{R}^n}(-\Delta)^{-s}\psi_1^{-}\cdot\psi_1^{+}dx}{\|(-\Delta)^{-\frac{s}{2}}\psi_1^{+}\|_{L^2(\mathbb{R}^n)}^{\frac{2n}{n+2s}}}$$
and
$$\mathcal{S}_{s,n}^{\frac{n}{n+2s}}\leq \frac{\|\psi_1^{-}\|_{L^{\frac{2n}{n+2s}}(\mathbb{R}^n)}^{\frac{2n}{n+2s}}}{\|(-\Delta)^{-\frac{s}{2}}\psi_1^{-}\|_{L^2(\mathbb{R}^n)}^{\frac{2n}{n+2s}}}
=\frac{\|(-\Delta)^{-\frac{s}{2}}\psi_1^{-}\|_{L^2(\mathbb{R}^n)}^2-\int_{\mathbb{R}^n}(-\Delta)^{-s}\psi_1^{+}\cdot\psi_1^{-}dx}{\|(-\Delta)^{-\frac{s}{2}}\psi_1^{-}\|_{L^2(\mathbb{R}^n)}^{\frac{2n}{n+2s}}},$$
which implies that $$\|(-\Delta)^{-\frac{s}{2}}\psi_1^{+}\|_{L^2(\mathbb{R}^n)}^2-\int_{\mathbb{R}^n}(-\Delta)^{-s}\psi_1^{-}\cdot\psi_1^{+}dx\geq \mathcal{S}_{s,n}^{\frac{n}{2s}}$$
and
$$\|(-\Delta)^{-\frac{s}{2}}\psi_1^{-}\|_{L^2(\mathbb{R}^n)}^2-\int_{\mathbb{R}^n}(-\Delta)^{-s}\psi_1^{+}\cdot\psi_1^{-}dx\geq \mathcal{S}_{s,n}^{\frac{n}{2s}}.$$
Testing the equation (\ref{el equ}) with $\psi_1$, we derive that
\begin{equation}\begin{split}
&\|\psi_1\|_{L^{\frac{2n}{n+2s}}(\mathbb{R}^n)}^{\frac{2n}{n+2s}}\\
&\ \ =\int_{\mathbb{R}^n}(-\Delta)^{-s}\psi_1\cdot \psi_1dx\\
&\ \ =\|(-\Delta)^{-\frac{s}{2}}\psi_1^{+}\|_{L^2(\mathbb{R}^n)}^2+\|(-\Delta)^{-\frac{s}{2}}\psi_1^{-}\|_{L^2(\mathbb{R}^n)}^2-\int_{\mathbb{R}^n}(-\Delta)^{-s}\psi_1^{-}\cdot\psi_1^{+}-\int_{\mathbb{R}^n}(-\Delta)^{-s}\psi_1^{+}\cdot\psi_1^{-}dx.
\end{split}\end{equation}
Combining the above estimate, we conclude that $\|\psi_1\|_{L^{\frac{2n}{n+2s}}(\mathbb{R}^n)}^{\frac{2n}{n+2s}}\geq 2\mathcal{S}_{s,n}^{\frac{n}{2s}},$
which is a contradiction with the assumption on the integral $\int_{\mathbb{R}^n}|f_k|^{\frac{2n}{n+2s}}dx$. Hence $\psi_1$ does not change signs and it must belong to $\mathcal{M}_{HLS}$. Without loss of generality, we denote $\big(h_k^{1}\big)^{-\frac{n+2s}{2}}\psi_1(\frac{\cdot-x_k^j}{h_k^j})$ by $U_{H}(z_k, \lambda_k)$. Then it follows that $$\lim\limits_{k\rightarrow +\infty}\|(-\Delta)^{-\frac{s}{2}}\Big( f_k-U_{H}(z_k,\lambda_k)\Big)\|_{L^2(\mathbb{R}^n)}=0.$$

Finally, we will show that $$\lim\limits_{k\rightarrow +\infty}\Big(|f_k|^{-\frac{4s}{n+2s}}f_k-(U_{H}(z_k,\lambda_k))^{-\frac{4s}{n+2s}}U_{H}(z_k,\lambda_k), f_k-U_{H}(z_k,\lambda_k)\Big)=0.$$
Direct computation gives that
\begin{equation}\begin{split}
&\Big(|f_k|^{-\frac{4s}{n+2s}}f_k-(U_{H}(z_k,\lambda_k))^{-\frac{4s}{n+2s}}U_{H}(z_k,\lambda_k), f_k-U_{H}(z_k,\lambda_k)\Big)\\
&=\Big(|f_k|^{-\frac{4s}{n+2s}}f_k, f_k-U_{H}(z_k,\lambda_k)\Big)-\Big((U_{H}(z_k,\lambda_k))^{-\frac{4s}{n+2s}}U_{H}(z_k,\lambda_k), f_k-U_{H}(z_k,\lambda_k)\Big)\\
&=\Big(|f_k|^{-\frac{4s}{n+2s}}f_k-(-\Delta)^{-s}f_k, f_k-U_{H}(z_k,\lambda_k)\Big)+\Big(-\Delta)^{-s}f_k, f_k-U_{H}(z_k,\lambda_k)\Big)\\
&\ \ -\Big((U_{H}(z_k,\lambda_k))^{-\frac{4s}{n+2s}}U_{H}(z_k,\lambda_k), f_k-U_{H}(z_k,\lambda_k)\Big)\\
&:=I+II+III.
\end{split}\end{equation}
For $I$, we derive that
$$I_1\leq \left\||f_k|^{-\frac{4s}{n+2s}}f_k-(-\Delta)^{-s}f_k\right\|_{L^{\frac{2n}{n-2s}}(\mathbb{R}^n)}\left(\|f_k\|_{L^{\frac{2n}{n+2s}}(\mathbb{R}^n)}+\|U_H\|_{L^{\frac{2n}{n+2s}}(\mathbb{R}^n)}\right)\rightarrow 0$$
as $k\rightarrow +\infty$ since $\lim\limits_{k\rightarrow +\infty}\left\||f_k|^{-\frac{4s}{n+2s}}f_k-(-\Delta)^{-s}f_k\right\|_{L^{\frac{2n}{n-2s}}(\mathbb{R}^n)}=0$ through the assumption.
For $II$, we derive that
$$II\leq \|(-\Delta)^{-\frac{s}{2}}f_k\|_{L^2(\mathbb{R}^n)}\|(-\Delta)^{-\frac{s}{2}}\left(f_k-U_{H}(z_k,\lambda_k\right)\|_{L^2(\mathbb{R}^n)}\rightarrow 0$$
as $k\rightarrow +\infty$ since we have obtained $\lim\limits_{k\rightarrow +\infty}\|(-\Delta)^{-\frac{s}{2}}\left(f_k-U_{H}(z_k,\lambda_k\right)\|_{L^2(\mathbb{R}^n)}=0$.
For $III$, since $U_{H}(z_k,\lambda_k)$ satisfies equation
$$(U_{H}(z_k,\lambda_k))^{-\frac{4s}{n+2s}}U_{H}(z_k,\lambda_k)=(-\Delta)^{-s}(U_{H}(z_k,\lambda_k)),$$ it follows that
\begin{equation}\begin{split}
&\Big((U_{H}(z_k,\lambda_k))^{-\frac{4s}{n+2s}}U_{H}(z_k,\lambda_k), f_k-U_{H}(z_k,\lambda_k)\Big)\\
&\ \ =\Big((-\Delta)^{-s}(U_{H}(z_k,\lambda_k)), f_k-U_{H}(z_k,\lambda_k)\Big)\\
&\ \ \leq \|(-\Delta)^{-\frac{s}{2}}(U_{H}(z_k,\lambda_k))\|_{L^2(\mathbb{R}^n)}\|(-\Delta)^{-\frac{s}{2}}\left(f_k-U_{H}(z_k,\lambda_k\right)\|_{L^2(\mathbb{R}^n)}\rightarrow 0
\end{split}\end{equation}
as $k\rightarrow +\infty$. This proves
\begin{equation}\label{est deHLS}
\lim\limits_{k\rightarrow +\infty}\Big(|f_k|^{-\frac{4s}{n+2s}}f_k-(U_{H}(z_k,\lambda_k))^{-\frac{4s}{n+2s}}U_{H}(z_k,\lambda_k), f_k-U_{H}(z_k,\lambda_k)\Big)=0.
\end{equation}
Finally, we prove $$\lim_{k\rightarrow +\infty}\int_{\mathbb{R}^n}|f_k-U_{H}|^{\frac{2n}{n+2s}}dx=0.$$
Write $$\langle|f_k|^{\frac{2n}{n+2s}-2}f_k-|U_{H}|^{\frac{2n}{n+2s}-2}U_{H}, f_k-U_{H}\rangle:=I+II,$$
where
$$I=\int_{\{f_k\geq 0\}}(|f_k|^{\frac{2n}{n+2s}-2}f_k-|U_{H}|^{\frac{2n}{n+2s}-2}U_{H})(f_k-U_{H})dx$$
and
$$II=\int_{\{f_k< 0\}}(|f_k|^{\frac{2n}{n+2s}-2}f_k-|U_{H}|^{\frac{2n}{n+2s}-2}U_{H})(f_k-U_{H})dx.$$
For $I$, using the inequality $a^q-b^q\geq q a^{q-1}(a-b)$ when $a\geq b\geq0,~~0<q<1$, H$\ddot{o}$lder's inequality for index less than 1, we can write
\begin{equation}\begin{split}\label{est I}
 &\ \ I\geq \frac{n-2s}{n+2s}\left(\int_{\{f_k\geq U_{H}\}} f_k^{\frac{2n}{n+2s}-2}(f_k-U_{H})^2dx+ \int_{\{0\leq f_k\leq  U_{H}\}}U_{H}^{\frac{2n}{n+2s}-2}(U_{H}-f_k)^2dx\right)\\
 &\ \ \geq \frac{n-2s}{n+2s} \left(\int_{\{f_k\geq U_{H}\}} |f_k|^{\frac{2n}{n+2s}}\right)^{-\frac{2s}{n}}\|(f_k-U_{H})\chi_{\{f_k\geq U_{H}\}}\|_{\frac{2n}{n+2s}}^2\\
 &\ \ \ \ \ +\frac{n-2s}{n+2s}\left(\int_{\{0\leq f_k\leq U_{H}\}} |U_{H}|^{\frac{2n}{n+2s}}\right)^{-\frac{2s}{n}}\|(f_k-U_{H})\chi_{\{0\leq f_k\leq U_{H}\}}\|_{\frac{2n}{n+2s}}^2\\
 &\ \ \geq \min\left\{\left(\int_{\mathbb{R}^n} |f_k|^{\frac{2n}{n+2s}}\right)^{-\frac{2s}{n}},\  \left(\int_{\mathbb{R}^n} |U_{H}|^{\frac{2n}{n+2s}}\right)^{-\frac{2s}{n}}\right\} \frac{n-2s}{n+2s}2^{-2s/n}\|(f_k-U_{H})\chi_{\{f_k\geq 0\}}\|_{\frac{2n}{n+2s}}^2.
 \end{split}\end{equation}
For $II$, using the inequality $a^q+b^q\geq  a^{q-1}(a+b)$ when $a\geq b\geq0,~~0<q<1$, H$\ddot{o}$lder's inequality for index less than 1, we can write
\begin{equation}\begin{split}\label{est II}
 &II=\int_{\{f_k< 0\}}(|f_k|^{\frac{2n}{n+2s}-2}(-f_k)+|U_{H}|^{\frac{2n}{n+2s}-2}U_{H})((-f_k)+U_{H})dx.\\
 &\ \ \geq \left(\int_{\{-f_k\geq U_{H}\}} f_k^{\frac{2n}{n+2s}-2}(f_k-U_{H})^2dx+ \int_{\{0\leq-f_k\leq  U_{H}\}}U_{H}^{\frac{2n}{n+2s}-2}(U_{H}-f_k)^2dx\right)\\
 &\ \ \geq  \left(\int_{\{-f_k\geq U_{H}\}} |f_k|^{\frac{2n}{n+2s}}\right)^{-\frac{2s}{n}}\|(f_k-U_{H})\chi_{\{-f_k\geq U_{H}\}}\|_{\frac{2n}{n+2s}}^2\\
 &\ \ \ \ \ +\left(\int_{\{0< -f_k\leq U_{H}\}} |U_{H}|^{\frac{2n}{n+2s}}\right)^{-\frac{2s}{n}}\|(f_k-U_{H})\chi_{\{0< -f_k\leq U_{H}\}}\|_{\frac{2n}{n+2s}}^2\\
 &\ \ \geq \min\left\{\left(\int_{\mathbb{R}^n} |f_k|^{\frac{2n}{n+2s}}\right)^{-\frac{2s}{n}},\  \left(\int_{\mathbb{R}^n} |U_{H}|^{\frac{2n}{n+2s}}\right)^{-\frac{2s}{n}}\right\} 2^{-2s/n}\|(f_k-U_{H})\chi_{\{f_k<0\}}\|_{\frac{2n}{n+2s}}^2.
 \end{split}\end{equation}
Along with (\ref{est deHLS}), (\ref{est I}) and (\ref{est II}), we accomplish the proof of lemma \ref{deHLS}.
\end{proof}

Next, we will set up the Struwe type decomposition lemma for critical points of Sobolev inequality without the nonnegativity assumption for functions under consideration from the analogues for the HLS inequality i.e. Lemma \ref{deHLS} by a duality method. First we need the following Lemma.
\begin{lemma}\label{simh}
Assume that $u_k\in \dot{H}^{s}(\mathbb{R}^n)$ satisfying
$$\lim\limits_{k\rightarrow +\infty}\left\|(-\Delta)^su_k-|u_k|^{\frac{4s}{n-2s}}u_k\right\|_{\dot H^{-s}(\mathbb{R}^n)}=0.$$
If we let $f_k=|u_k|^{\frac{4s}{n-2s}}u_k$, then there holds
$$\lim\limits_{k\rightarrow +\infty}\left\|(-\Delta)^{-s}f_k-|f_k|^{-\frac{4s}{n+2s}}f_k\right\|_{L^{\frac{2n}{n-2s}}(\mathbb{R}^n)}=0.$$
\end{lemma}
\begin{proof}
Let $F(u_k)=(-\Delta)^su_k-|u_k|^{\frac{4s}{n-2s}}u_k$, then direct calculation leads to
$$|f_k|^{-\frac{4s}{n+2s}}f_k-(-\Delta)^{-s}f_k=(-\Delta)^{-s}(F(u_k)).$$
Then it follows that
\begin{equation}\label{shls}\begin{split}
\left\|(-\Delta)^su_k-|u_k|^{\frac{4s}{n-2s}}u_k\right\|_{H^{-s}(\mathbb{R}^n)}&=\|F(u_k)\|_{H^{-s}(\mathbb{R}^n)}\\
&=\|(-\Delta)^{-\frac{s}{2}}(F(u_k))\|_{L^2(\mathbb{R}^n)}\\
&=\|(-\Delta)^{\frac{s}{2}}\Big((-\Delta)^{-s}(F(u_k))\Big)\|_{L^2(\mathbb{R}^n)}\\
&\geq \mathcal{S}_{s,n}\left\||f_k|^{-\frac{4s}{n+2s}}f_k-(-\Delta)^{-s}f_k\right\|_{L^{\frac{2n}{n+2s}}(\mathbb{R}^n)}.
\end{split}\end{equation}
This accomplishes the proof of Lemma \ref{simh}.
\end{proof}

\begin{lemma}\label{deSo}
For $0<s<\frac{n}{2}$ and $\{u_k\}_k\in \dot{H}^s(\mathbb{R}^n)$ such that
$$\big(\frac{1}{2}\big)^{\frac{n-2s}{n}}\mathcal{S}_{s,n}^{\frac{n}{2s}}\leq \int_{\mathbb{R}^n}|(-\Delta)^{\frac{s}{2}}u_k|^2dx\leq \big(\frac{3}{2}\big)^{\frac{n-2s}{n}}\mathcal{S}_{s,n}^{\frac{n}{2s}},$$
$$\lim\limits_{k\rightarrow +\infty}\left\|(-\Delta)^su_k-|u_k|^{\frac{4s}{n-2s}}u_k\right\|_{\dot H^{-s}(\mathbb{R}^n)}=0,$$
then there exists a subsequence of $\{u_k\}_k$ (still denoted by $\{u_k\}_k$) and $\{z_k\}_k\subset \mathbb{R}^n$, $\{\lambda_k\}_k\subset \mathbb{R}^{+}$ such that
$$\lim\limits_{k\rightarrow +\infty}\left\|(-\Delta)^{\frac{s}{2}}(u_k-\beta U(z_k, \lambda_k))\right\|_{L^2(\mathbb{R}^n)}=0,$$
where $\beta$ is equal to $+1$ or $-1$.
\end{lemma}
\begin{proof}
Let $f_k=|u_k|^{\frac{4s}{n-2s}}u_k$, then through the fractional Sobolev inequality, we can get
$$\int_{\mathbb{R}^n}|f_k|^{\frac{2n}{n+2s}}dx=\int_{\mathbb{R}^n}|u_k|^{\frac{2n}{n-2s}}dx\leq (\mathcal{S}_{s,n}^{-1})^{\frac{n}{n-2s}} \big(\int_{\mathbb{R}^n}|(-\Delta)^{\frac{s}{2}}u_k|^2dx\big)^{\frac{n}{n-2s}}.$$
This together with $\lim\limits_{k\rightarrow +\infty}\int_{\mathbb{R}^n}|(-\Delta)^{\frac{s}{2}}u_k|^2dx\leq \big(\frac{3}{2}\big)^{\frac{n-2s}{n}}\mathcal{S}_{s,n}^{\frac{n}{2s}}$ gives that
$$\lim\limits_{k\rightarrow +\infty}\int_{\mathbb{R}^n}|f_k|^{\frac{2n}{n+2s}}dx\leq \frac{3}{2}\mathcal{S}_{s,n}^{\frac{n}{2s}}.$$
On the other hand, since $$\lim\limits_{k\rightarrow +\infty}\left\|(-\Delta)^su_k-|u_k|^{\frac{4s}{n-2s}}u_k\right\|_{H^{-s}(\mathbb{R}^n)}=0,$$ we derive that
$$\lim\limits_{k\rightarrow +\infty}|\big((-\Delta)^su_k-|u_k|^{\frac{4s}{n-2s}}u_k, u_k\big)|\leq \lim\limits_{k\rightarrow +\infty}\left\|(-\Delta)^su_k-|u_k|^{\frac{4s}{n-2s}}u_k\right\|_{\dot H^{-s}(\mathbb{R}^n)} \|u_k\|_{\dot H^s(\mathbb{R}^n)}=0.$$
This gives that $$\lim\limits_{k\rightarrow +\infty}\int_{\mathbb{R}^n}|(-\Delta)^{\frac{s}{2}}u_k|^2dx=\lim\limits_{k\rightarrow +\infty}\int_{\mathbb{R}^n}|u_k|^{\frac{2n}{n-2s}}dx.$$
Hence $$\lim\limits_{k\rightarrow +\infty}\int_{\mathbb{R}^n}|f_k|^{\frac{2n}{n+2s}}dx=\lim\limits_{k\rightarrow +\infty}\int_{\mathbb{R}^n}|u_k|^{\frac{2n}{n-2s}}dx=\lim\limits_{k\rightarrow +\infty}\int_{\mathbb{R}^n}|(-\Delta)^{\frac{s}{2}}u_k|^2dx\geq \big(\frac{1}{2}\big)^{\frac{n-2s}{n}}\mathcal{S}_{s,n}^{\frac{n}{2s}}.$$
Now, applying Lemma \ref{deHLS}(see remark \ref{remark}), we derive that $$\lim\limits_{k\rightarrow +\infty}\Big(|f_k|^{-\frac{4s}{n+2s}}f_k-\beta U_{H}^{-\frac{4s}{n+2s}}(z_k,\lambda_k)U_{H}(z_k,\lambda_k), f_k-\beta U_{H}(z_k,\lambda_k)\Big)=0.$$
Since $f_k=|u_k|^{\frac{4s}{n-2s}}u_k$, then it follows that
\begin{equation}\begin{split}\label{eq 1}
&\Big(|f_k|^{-\frac{4s}{n+2s}}f_k-\beta U_{H}^{-\frac{4s}{n+2s}}(z_k,\lambda_k)U_{H}(z_k,\lambda_k), f_k-\beta U_{H}(z_k,\lambda_k)\Big)\\
&\ \ =\Big(u_k-\beta U(z_k,\lambda_k), |u_k|^{\frac{2n}{n-2s}-2}u_k-\beta U^{\frac{2n}{n-2s}-2}(z_k,\lambda_k)U(z_k,\lambda_k)\Big),
\end{split}\end{equation}
where we used the fact $U_{H}^{-\frac{4s}{n+2s}}(z_k,\lambda_k)U_{H}(z_k,\lambda_k)=U(z_k, \lambda_k)$.
Combining (\ref{eq 1}) and the vector inequality $$\big(|\overrightarrow{a}|^{p-2}\overrightarrow{a}-|\overrightarrow{b}|^{p-2}, \overrightarrow{a}-\overrightarrow{b}\big)\geq C_p |\overrightarrow{a}-\overrightarrow{b}|^p,\ \ \forall p\geq 2,$$
we conclude that
$$\lim\limits_{k\rightarrow +\infty}\int_{\mathbb{R}^n}|u_k-\beta U(z_k,\lambda_k)|^{\frac{2n}{n-2s}}dx=0.$$
Now, we start to prove that
$$\lim\limits_{k\rightarrow +\infty}\int_{\mathbb{R}^n}|(-\Delta)^{\frac{s}{2}}(u_k-\beta U(z_k, \lambda_k))|^2dx=0.$$
We can write
\begin{equation}\begin{split}
&\|(-\Delta)^{\frac{s}{2}}(u_k-\beta U(z_k, \lambda_k))\|_{L^2(\mathbb{R}^n)}\\
&\ \ =\Big((-\Delta)^su_k, u_k-\beta U(z_k, \lambda_k)\Big)-\Big((-\Delta)^s(\beta U(z_k, \lambda_k)), u_k-\beta U(z_k, \lambda_k)\Big)\\
&\ \ =\Big((-\Delta)^su_k-|u_k|^{\frac{4s}{n-2s}}u_k,u_k-\beta U(z_k, \lambda_k)\Big)+\Big(|u_k|^{\frac{4s}{n-2s}}u_k,u_k-\beta U(z_k, \lambda_k)\Big)\\
&\ \ \ \ -\Big((-\Delta)^s(\beta U(z_k, \lambda_k)), u_k-\beta U(z_k, \lambda_k)\Big)\\
&\ \ :=I+II+III.
\end{split}\end{equation}
For $I$, we derive that
$$I\leq \left\|(-\Delta)^su_k-|u_k|^{\frac{4s}{n-2s}}u_k\right\|_{\dot H^{-s}(\mathbb{R}^n)}\|u_k-\beta U(z_k, \lambda_k)\|_{\dot H^s(\mathbb{R}^n)}\rightarrow 0$$
as $k\rightarrow +\infty$ since $\lim\limits_{k\rightarrow +\infty}\left\|(-\Delta)^su_k-|u_k|^{\frac{4s}{n-2s}}u_k\right\|_{\dot H^{-s}(\mathbb{R}^n)}=0$ through the assumption of Lemma \ref{deSo}.
For $II$, we derive that
$$II\leq \|u_k\|_{L^{\frac{2n}{n-2s}}(\mathbb{R}^n)}^{\frac{2n}{n-2s}-1}\|u_k-\beta U(z_k, \lambda_k)\|_{L^{\frac{2n}{n-2s}}(\mathbb{R}^n)}\rightarrow 0$$
as $k\rightarrow +\infty$ since we have obtained $\lim\limits_{k\rightarrow +\infty}\|u_k-\beta U(z_k, \lambda_k)\|_{L^{\frac{2n}{n-2s}}(\mathbb{R}^n)}=0$.
For III, since $U(z_k, \lambda_k)$ satisfies equation $$(-\Delta)^s U(z_k, \lambda_k)=U(z_k, \lambda_k)^{\frac{4s}{n-2s}}U(z_k, \lambda_k),$$
then it follows that
\begin{equation}\begin{split}
III&=\Big((\beta U(z_k, \lambda_k))^{\frac{4s}{n-2s}}U(z_k, \lambda_k), u_k-\beta U(z_k, \lambda_k)\Big)\\
&\leq \|U(z_k,\lambda_k)\|_{L^{\frac{2n}{n-2s}}(\mathbb{R}^n)}^{\frac{2n}{n-2s}-1}\|u_k-\beta U(z_k, \lambda_k)\|_{L^{\frac{2n}{n-2s}}(\mathbb{R}^n)}\rightarrow 0
\end{split}\end{equation}
as $k\rightarrow +\infty$. Then we accomplish the proof of Lemma \ref{deSo}.
\end{proof}

\section{Proof of the main Theorem}
In this section, we will prove the stability for critical point of the Hardy-Littlewood-Sobolev inequality. We first prove the following local stability.

\begin{lemma}\label{loc sta}
For all $f\in L^{\frac{2n}{n+2s}}(\mathbb{R}^n)$ with
$$\frac{1}{2}\mathcal{S}_{s,n}^{\frac{n}{2s}}\leq\int_{\mathbb{R}^n}|f|^{\frac{2n}{n+2s}}dx\leq \frac{3}{2}\mathcal{S}_{s,n}^{\frac{n}{2s}},$$
 there holds
$$\left\||f|^{-\frac{4s}{n+2s}}f-(-\Delta)^{-s}f\right\|_{L^\frac{2n}{n-2s}(\mathbb{R}^n)}\geq C_{n,s}d(f,\mathcal{M}_{HLS})+o(d(f,\mathcal{M}_{HLS})),$$
where $d(f,\mathcal{M}_{HLS})=\inf\limits_{h\in \mathcal{M}_{HLS}}\|f-h\|_{L^{\frac{2n}{n+2s}}(\mathbb{R}^n)}$.
\end{lemma}

In fact, we only need to prove the analogue Lemma on the sphere $\mathbb{S}^n$. By the stereographic projection $\mathcal{S}:\mathbb{R}^n\cup\{\infty\}\rightarrow \mathbb{S}^n$, we know that $\mathbb{R}^n$ (or rather $\mathbb{R}^n\cup\{\infty\}$) and $\mathbb{S}^n$ ($\subset \mathbb{R}^{n+1}$) are conformally equivalent. For any $f\in L^{\frac{2n}{n+2s}}(\mathbb{R}^n)$, let $u$ be defined on $\mathbb{S}^n$ by (See Lieb and Loss in \cite{LiebLoss})
$$u(\xi)=\left(\frac{1+|\mathcal{S}^{-1}(\xi)|^2}{2}\right)^{\frac{n+2s}{2}}f(\mathcal{S}^{-1}(\xi)),$$
then $$\|f\|_{L^{\frac{2n}{n+2s}}(\mathbb{R}^n)}=\|u\|_{L^{\frac{2n}{n+2s}}(\mathbb{S}^n)}.$$
Denote the fractional integral operator on the sphere by
$$\mathcal{P}_{2s}(u)(\omega)=\frac{\Gamma(\frac{n}{2}-s)}{2^{2s}\pi^{\frac{n}{2}}\Gamma(s)}\int_{\mathbb{S}^n}\frac{u(\xi)}{|\omega-\xi|^{n-2s}}d\xi.$$
Using the fact
$$J_{\mathcal{S}}(x)^{\frac{1}{n}}|x-y|^2J_{\mathcal{S}}(y)^{\frac{1}{n}}=|\mathcal{S}(x)-\mathcal{S}(y)|,~~~x,y\in \mathbb{R}^n$$
we easily see that
$$\mathcal{P}_{2s}(u)(\omega)=J_{\mathcal{S}^{-1}}^{\frac{n-2s}{2n}}(\omega)((-\Delta)^{-s}f)(\mathcal{S}^{-1}\omega).$$
Thus the sharp Hardy-Littlewood-Sobolev inequality on the sphere $\mathbb{S}^n$ states that
\begin{equation*}
	\| u\|_{L^{\frac{2n}{n+2s}}(\mathbb{S}^n)}^2\geq \mathcal S_{s,n} \langle\mathcal{P}_{2s}u, u\rangle
	\qquad\text{for all}\ u\in L^{\frac{2n}{n+2s}}(\mathbb{S}^n)
\end{equation*}
and the equality holds if and only if
$$u\in M:=\left\{c\left(\frac{\sqrt{1-|\xi|^2}}{1-\xi\cdot\omega}\right)^{\frac{n+2s}{2}}: c\in \mathbb{R},\  \xi\in \mathbb{R}^{n},\  |\xi|<0 \right\}.$$

Together with the fact that $|u(\omega)|^{\frac{-4s}{n+2s}}u(\omega)=J_{\mathcal{S}^{-1}}^{\frac{n-2s}{2n}}(\omega)|f(\mathcal{S}^{-1}\omega)|^{\frac{-4s}{n+2s}}f(\mathcal{S}^{-1}\omega),$
it is well known that all constant-sign solutions of the Euler-Lagrange equation related to the Hardy-Littlewood-Sobolev inequality on the sphere
$$\mathcal{P}_{2s}(u)=|u|^{-\frac{4s}{n+2s}}u$$
are in the space
$$\mathcal{M}:=\left\{\pm c_{n,s}\left(\frac{\sqrt{1-|\xi|^2}}{1-\xi\cdot\omega}\right)^{\frac{n+2s}{2}}: c\in \mathbb{R},\  \xi\in \mathbb{R}^{n},\  |\xi|<0 \right\},$$
where $$c_{n,s}=\left(\frac{\Gamma(\frac{n}{2}+s)}{\Gamma(\frac{n}{2}-s)}\right)^{\frac{n+2s}{4s}}.$$
Then it is easy to check that
$$d(u,\mathcal{M})=\inf\limits_{g\in \mathcal{M}}\|u-g\|_{L^{\frac{2n}{n+2s}}(\mathbb{S}^n)}=\inf\limits_{h\in \mathcal{M}_{HLS}}\|f-h\|_{L^{\frac{2n}{n+2s}}(\mathbb{R}^n)}=d(f,\mathcal{M}_{HLS}).$$
and
 $$\frac{\left\||f|^{-\frac{4s}{n+2s}}f-(-\Delta)^{-s}f\right\|_{L^\frac{2n}{n-2s}(\mathbb{R}^n)}}{d(f,\mathcal{M}_{HLS})}
=\frac{\left\||u|^{-\frac{4s}{n+2s}}u-\mathcal{P}_{2s}(u)\right\|_{L^{\frac{2n}{n-2s}}(\mathbb{S}^n)}}{d(u,\mathcal{M})},$$
which implies the following stabilities for critical points of HLS inequalities on $\mathbb{S}^n$ are equivalence to Lemma \ref{loc sta}.
\begin{lemma}\label{loc sta sph}
For all $u\in L^{\frac{2n}{n+2s}}(\mathbb{S}^n)$ with
$$\frac{1}{2}\mathcal{S}_{s,n}^{\frac{n}{2s}}\leq\|u\|^{\frac{2n}{n+2s}}_{L^{\frac{2n}{n+2s}}(\mathbb{S}^n)}\leq \frac{3}{2}\mathcal{S}_{s,n}^{\frac{n}{2s}},$$
then there holds
$$\left\||u|^{-\frac{4s}{n+2s}}u-\mathcal{P}_{2s}(u)\right\|_{L^{\frac{2n}{n-2s}}(\mathbb{S}^n)}\geq C_{n,s}d(u,\mathcal{M})+o(d(u,\mathcal{M})),$$
where $d(u,\mathcal{M})=\inf\limits_{g\in \mathcal{M}}\|u-g\|_{L^{\frac{2n}{n+2s}}(\mathbb{S}^n)}$ and
$$C_{n,s}=\min\{\frac{\Gamma(\frac{n}{2}-s+1)}{\Gamma(\frac{n}{2}+s+1)}\frac{s}{n/2+s+1}
|\mathbb{S}^n|^{-\frac{2s}{n}}\sqrt{\frac{2s}{n+2s}},\ \frac{\Gamma(n/2-s)}{\Gamma(n/2+s)}\frac{2s}{n+2s}|\mathbb{S}^n|^{-\frac{2s}{n}}\}$$
\end{lemma}
\medskip

The difficulty to prove the local stability Lemma \ref{loc sta sph} is the lack of Hilbert space structure for the distance $d(u,\mathcal{M})$, which prevents us from using the orthogonal property of spherical harmonics. To overcome the difficult, We will adopt the idea from our earlier work \cite{CLT2} to prove the optimal local stability of HLS inequalities, where we encounter the same difficulty. In fact, assume $u\in L^{\frac{2n}{n+2s}}(\mathbb{S}^n)$ with $\inf\limits_{h\in \mathcal{M}}\|u-h\|_{L^{\frac{2n}{n+2s}}(\mathbb{S}^n)}=o(1)$ and we will decompose $u$ in $H^{-s}(\mathbb{S}^n)$ instead of in $L^{\frac{2n}{n+2s}}(\mathbb{S}^n)$. More precisely, let $\varphi\in \mathcal{M}$ (see Lemma \ref{decompostition} for the existence of $\varphi$) such that
$$\inf\limits_{g\in \mathcal{M}}\langle\mathcal{P}_{2s}(u-g),(u-g)\rangle=\langle\mathcal{P}_{2s}(u-\varphi),(u-\varphi)\rangle,$$  then by the Hardy-Littlewood-Sobolev inequality, there holds
\begin{equation}\label{Sob dis}
\langle\mathcal{P}_{2s}(u-\varphi),(u-\varphi)\rangle\leq \mathcal{S}_{s,n}^{-1} \inf\limits_{h\in \mathcal{M}}\|u-h\|^2_{L^{\frac{2n}{n+2s}}(\mathbb{S}^n)}=o(1).
\end{equation}
It is obvious that $\|u-\varphi\|_{L^{\frac{2n}{n+2s}}(\mathbb{S}^n)}\geq \mathcal{S}_{s,n}^{1/2}\|(-\Delta)^{-s/2}(u-\varphi)\|_2=o(1)$. We will prove $\|u-\varphi\|_{L^{\frac{2n}{n+2s}}(\mathbb{S}^n)}=o(1)$. In fact, we will prove $\|u-\varphi\|_{L^{\frac{2n}{n+2s}}(\mathbb{S}^n)}$ and $\inf\limits_{h\in \mathcal{M}}\|u-h\|_{L^{\frac{2n}{n+2s}}(\mathbb{S}^n)}$ are infinitesimals of the same order.

\begin{lemma}\label{infinitesimals}
With above notations, there holds,
$$\|u-\varphi\|_{L^{\frac{2n}{n+2s}}(\mathbb{S}^n)}=o(1).$$
And more precisely,
$$\inf\limits_{h\in \mathcal{M}}\|u-h\|_{L^{\frac{2n}{n+2s}}(\mathbb{S}^n)}\leq \|u-\varphi\|_{L^{\frac{2n}{n+2s}}(\mathbb{S}^n)}\leq \left(1+2\mathcal{S}_{s,n}\sqrt{\frac{2(n+2s)}{n-2s}}\right)\inf\limits_{h\in \mathcal{M}}\|u-h\|_{L^{\frac{2n}{n+2s}}(\mathbb{S}^n)}$$
if $\inf\limits_{h\in \mathcal{M}}\|u-h\|_{L^{\frac{2n}{n+2s}}(\mathbb{S}^n)}=o(1)$.
\end{lemma}

\begin{proof}

By Lemma \ref{decompostition} there esists $\phi\in \mathcal{M}$ such that
$$\inf\limits_{h\in \mathcal{M}}\|u-h\|_{L^{\frac{2n}{n+2s}}(\mathbb{S}^n)}=\|u-\phi\|_{L^{\frac{2n}{n+2s}}(\mathbb{S}^n)}=o(1),$$
and since
\begin{equation}\label{est 1}
\|u-\varphi\|_{L^{\frac{2n}{n+2s}}(\mathbb{S}^n)}\leq \|u-\phi\|_{L^{\frac{2n}{n+2s}}(\mathbb{S}^n)}+\|\phi-\varphi\|_{L^{\frac{2n}{n+2s}}(\mathbb{S}^n)},
\end{equation}
then
we only need to prove that
$$\|\phi-\varphi\|_{L^{\frac{2n}{n+2s}}(\mathbb{S}^n)}=o(1).$$
First, we can derive that
$$\langle\mathcal{P}_{2s}(\phi-\varphi),(\phi-\varphi)\rangle^{1/2}=o(1)$$
from (\ref{Sob dis}) and
\begin{align}\label{est 2}\nonumber
& \langle\mathcal{P}_{2s}(\phi-\varphi),(\phi-\varphi)\rangle^{1/2}\\
& \leq \langle\mathcal{P}_{2s}(u-\varphi),(u-\varphi)\rangle^{1/2}+\langle\mathcal{P}_{2s}(u-\phi),(u-\phi)\rangle^{1/2}
 \leq 2\mathcal{S}_{s,n}^{-1/2}\|u-\phi\|_{L^{\frac{2n}{n+2s}}(\mathbb{S}^n)}.
\end{align}
We now verify that $\phi$ and $\varphi$ have the same sign. Write
$\phi=\sigma_\phi b_\phi$ and $\varphi=\sigma_\varphi b_\varphi$, where
$\sigma_\phi,\sigma_\varphi\in\{-1,1\}$ and $b_\phi,b_\varphi$ are positive
normalized bubbles. Every such bubble $b$ satisfies
$$
\langle\mathcal P_{2s}b,b\rangle
=\int_{\mathbb S^n}b^{\frac{2n}{n+2s}}\,d\omega
=\mathcal S_{s,n}^{\frac n{2s}}.
$$
If $\sigma_\phi=-\sigma_\varphi$, then the strict positivity of the integral
kernel of $\mathcal P_{2s}$ gives
\begin{align*}
\langle\mathcal P_{2s}(\phi-\varphi),\phi-\varphi\rangle
&=\langle\mathcal P_{2s}(b_\phi+b_\varphi),b_\phi+b_\varphi\rangle\\
&=2\mathcal S_{s,n}^{\frac n{2s}}
+2\langle\mathcal P_{2s}b_\phi,b_\varphi\rangle
>2\mathcal S_{s,n}^{\frac n{2s}}.
\end{align*}
On the other hand, by (\ref{est 2}), we have
\ref{decompostition} imply
$$
\langle\mathcal P_{2s}(\phi-\varphi),\phi-\varphi\rangle
\leq 4\mathcal S_{s,n}^{-1}
\|u-\phi\|_{L^{\frac{2n}{n+2s}}(\mathbb S^n)}^2
=o(1),
$$
which is a contradiction. Hence $\sigma_\phi=\sigma_\varphi$ and we may therefore
assume that both $\phi$ and $\varphi$ are positive. Since they satisfy the
Euler-Lagrange equation
\begin{equation*}
\mathcal{P}_{2s}v=|v|^{-\frac{4s}{n+2s}}v,
\end{equation*}
we have
\begin{align*}
& \langle\mathcal{P}_{2s}(\phi-\varphi),(\phi-\varphi)\rangle\\
& =\langle |\phi|^{\frac{2n}{n+2s}-2}\phi-|\varphi|^{\frac{2n}{n+2s}-2}\varphi,\phi-\varphi\rangle
=\langle \phi^{\frac{2n}{n+2s}-1}-\varphi^{\frac{2n}{n+2s}-1},\phi-\varphi\rangle.
\end{align*}

Using the inequality $a^q-b^q\geq q a^{q-1}(a-b)$ when $a\geq b>0,~~0<q<1$, H$\ddot{o}$lder's inequality for index less than $1$, we have

\begin{equation}\begin{split}\label{est3*}
 &\langle \phi^{\frac{2n}{n+2s}-1}-\varphi^{\frac{2n}{n+2s}-1}, \phi-\varphi\rangle\\
 &\ \ \geq \frac{n-2s}{n+2s}\left(\int_{\{\phi\geq \varphi\}} \phi^{\frac{2n}{n+2s}-2}(\phi-\varphi)^2dx+ \int_{\{\phi\leq  \varphi\}} \varphi^{\frac{2n}{n+2s}-2}(\varphi-\phi)^2dx\right)\\
 &\ \ \geq \frac{n-2s}{n+2s} \left(\int_{\{\phi\geq \varphi\}} |\phi|^{\frac{2n}{n+2s}}\right)^{-\frac{2s}{n}}\|(\phi-\varphi)\chi_{\{\phi\geq \varphi\}}\|_{L^{\frac{2n}{n+2s}}(\mathbb{S}^n)}^2\\
 &\ \ \ \ \ +\frac{n-2s}{n+2s}\left(\int_{\{\phi\leq \varphi\}} |\varphi|^{\frac{2n}{n+2s}}\right)^{-\frac{2s}{n}}\|(\phi-\varphi)\chi_{\{\phi\leq \varphi\}}\|_{L^{\frac{2n}{n+2s}}(\mathbb{S}^n)}^2\\
 &\ \ \geq \min\left\{\left(\int_{\mathbb{S}^n} |\phi|^{\frac{2n}{n+2s}}\right)^{-\frac{2s}{n}},\  \left(\int_{\mathbb{S}^n} |\varphi|^{\frac{2n}{n+2s}}\right)^{-\frac{2s}{n}}\right\} \frac{n-2s}{n+2s}2^{-1}\|\phi-\varphi\|_{L^{\frac{2n}{n+2s}}(\mathbb{S}^n)}^2\\
 & =\mathcal{S}^{-1}_{s,n} \frac{n-2s}{n+2s}2^{-1}\|\phi-\varphi\|_{L^{\frac{2n}{n+2s}}(\mathbb{S}^n)}^2.
 \end{split}\end{equation}

 Combining the estimates (\ref{est 1}), (\ref{est 2}) and (\ref{est3*}), we derive that
 $$\|u-\varphi\|_{L^{\frac{2n}{n+2s}}(\mathbb{S}^n)}\leq \left(1+2\mathcal{S}_{s,n}\sqrt{\frac{2(n+2s)}{n-2s}}\right)\inf\limits_{h\in M}\|u-h\|_{L^{\frac{2n}{n+2s}}(\mathbb{S}^n)}$$ and Lemma \ref{infinitesimals} is proved.
\end{proof}

Now let us prove the local stability for critical points of the HLS inequality on $\mathbb{S}^n$.
\begin{proof}
First we claim that we only need to assume $u\in C(\mathbb{S}^n)$. For any $u\in L^{\frac{2n}{n+2s}}(\mathbb{S}^n)$, there exist $\{u_k\}\subset C(\mathbb{S}^n)$ such that $\|u_k-u\|_{L^{\frac{2n}{n+2s}}(\mathbb{S}^n)}\rightarrow 0$. By Lemma \ref{approxi} we have
$$\left\||u_k^{-\frac{4s}{n+2s}}u_k-\mathcal{P}_{2s}u_k\right\|_{L^{\frac{2n}{n-2s}}(\mathbb{S}^n)}
\rightarrow\left\||u|^{-\frac{4s}{n+2s}}u-\mathcal{P}_{2s}u\right\|_{L^{\frac{2n}{n-2s}}(\mathbb{S}^n)}.$$
Along with the fact $d(u_k,\mathcal{M})\rightarrow d(u,\mathcal{M})$, we can assume $u\in C(\mathbb{S}^n)$.

Next we claim that for any $v=\beta c_{n,s}+\epsilon r_0\in C(\mathbb{S}^n)$ with $\beta^{-\frac{4s}{n+2s}}>\frac{n/2-s+1}{n/2+s+1}$, $r_0$ satisfying $\|r_0\|_{L^{\frac{2n}{n+2s}}(\mathbb{S}^n)}>0$
and
$$\int_{\mathbb{S}^n}r_0d\omega=0=\int_{\mathbb{S}^n}r_0Y_{1,j}d\omega=0,~~j=1,\cdots,n+1$$
there holds
\begin{equation}\label{spherical exp}
\left\||\beta c_{n,s}+\epsilon r_0|^{-\frac{4s}{n+2s}}(\beta c_{n,s}+\epsilon r_0)-\mathcal{P}_{2s}(\beta c_{n,s}+\epsilon r_0)\right\|_{L^{\frac{2n}{n-2s}}(\mathbb{S}^n)}\geq C_{n,s,\beta}\epsilon\frac{\|r_0\|^2_{L^2(\mathbb{S}^n)}}{\|r_0\|_{L^{\frac{2n}{n+2s}}(\mathbb{S}^n)}}+o(\epsilon),
\end{equation}
where
$$C_{n,\beta,s}=\frac{\Gamma(\frac{n}{2}-s+1)}{\Gamma(\frac{n}{2}+s+1)}\left(\beta^{-\frac{4s}{n+2s}}-\frac{n/2-s+1}{n/2+s+1}\right).$$

Since
\begin{align*}
& \left\||\beta c_{n,s}+\epsilon r_0|^{-\frac{4s}{n+2s}}(\beta c_{n,s}+\epsilon r_0)-\mathcal{P}_{2s}(\beta c_{n,s}+\epsilon r_0)\right\|_{L^{\frac{2n}{n-2s}}(\mathbb{S}^n)}\\
& \geq \frac{1}{\|r_0\|_{L^{\frac{2n}{n+2s}}(\mathbb{S}^n)}}\int_{\mathbb{S}^n}[|\beta c_{n,s}+\epsilon r_0|^{-\frac{4s}{n+2s}}(\beta c_{n,s}+\epsilon r_0)-\mathcal{P}_{2s}(\beta c_{n,s}+\epsilon r_0)]r_0d\omega.
\end{align*}
Thus  expanding $r_0$ in terms of spherical harmonics $r_0=\sum_{l\geq 2}\sum_{m=1}^{N(n,l)}r_{l,m}Y_{l,m}$ with $r_{l,m}=\int_{\mathbb{S}^n}r_0Y_{l,m}d\omega$,
by Taylor expansion we have
\begin{align*}
& \int_{\mathbb{S}^n}[|\beta c_{n,s}+\epsilon r_0|^{-\frac{4s}{n+2s}}(\beta c_{n,s}+\epsilon r_0)]r_0d\omega\\
& =(\beta c_{n,s})^{\frac{2n}{n+2s}}\int_{\mathbb{S}^n}r_0d\omega+\epsilon\frac{n-2s}{n+2s}(\beta c_{n,s})^{-\frac{4s}{n+2s}}\int_{\mathbb{S}^n}r^2_0d\omega+o(\epsilon),\\
& =\epsilon\frac{n-2s}{n+2s}(\beta c_{n,s})^{-\frac{4s}{n+2s}}\sum_{l\geq 2}\sum_{m=1}^{N(n,l)}r^2_{l,m}+o(\epsilon)
\end{align*}
On the other hand, the Funk-Heck formula implies that if $Y$ is a spherical harmonic of degree $l$, then
$$\mathcal{P}_{2s}Y=\frac{\Gamma(l+n/2-s)}{\Gamma(l+n/2+s)}Y.$$
Thus
\begin{align*}
& \int_{\mathbb{S}^n}\mathcal{P}_{2s}(\beta c_{n,s}+\epsilon r_0)r_0d\omega=\epsilon\sum_{l\geq 2}\frac{\Gamma(l+n/2-s)}{\Gamma(l+n/2+s)}\sum_{m=1}^{N(n,l)}r^2_{l,m}.
\end{align*}
Recall that
$$c_{n,s}=\left(\frac{\Gamma(\frac{n}{2}+s)}{\Gamma(\frac{n}{2}-s)}\right)^{\frac{n+2s}{4s}},$$
then when $l\geq 2$ we have
\begin{align*}
& (\beta c_{n,s})^{-\frac{4s}{n+2s}}\frac{n-2s}{n+2s}-\frac{\Gamma(l+n/2-s)}{\Gamma(l+n/2+s)}\\
& =\frac{\Gamma(\frac{n}{2}-s)}{\Gamma(\frac{n}{2}+s)}\frac{n/2-s}{n/2+s}\left(\beta^{-\frac{4s}{n+2s}}-\frac{(n/2-s+1)\cdots(n/2-s+l-1)}{(n/2+s+1)\cdots(n/2+s+l-1)}\right)\\
& \geq \frac{\Gamma(\frac{n}{2}-s+1)}{\Gamma(\frac{n}{2}+s+1)}\left(\beta^{-\frac{4s}{n+2s}}-\frac{n/2-s+1}{n/2+s+1}\right).
\end{align*}
Thus we can prove the claim
\begin{align*}
& \left\||c_{n,s}+\epsilon r_0|^{-\frac{4s}{n+2s}}(c_{n,s}+\epsilon r_0)-\mathcal{P}_{2s}(c_{n,s}+\epsilon r_0)\right\|_{L^{\frac{2n}{n-2s}}(\mathbb{S}^n)}\\
& \geq
\epsilon\frac{\Gamma(\frac{n}{2}-s+1)}{\Gamma(\frac{n}{2}+s+1)}\left(\beta^{-\frac{4s}{n+2s}}-\frac{n/2-s+1}{n/2+s+1}\right)\frac{\|r_0\|^2_{L^2(\mathbb{S}^n)}}{\|r_0\|_{L^{\frac{2n}{n+2s}}(\mathbb{S}^n)}}+o(\epsilon).
\end{align*}

By Lemma \ref{decompostition}, we know that there exists $\phi\in \mathcal{M}$ such that
$$\inf\limits_{h\in \mathcal{M}}\langle\mathcal{P}_{2s}(u-h),u-h\rangle=\langle\mathcal{P}_{2s}(u-\phi),u-\phi\rangle,$$
where
$\phi=\pm c_{n,s}J_{\Phi}^{\frac{2n}{n+2s}}$ with $\Phi$ being a conformal transformation on the sphere and the function $r=u-\phi$ satisfies
$$r\bot span\{J_{\Phi}^{\frac{n+2s}{2n}}Y_{1,i}\circ\Phi,~~i=1,\cdots,n+1\},$$
in the "inner product" $\langle\mathcal{P}_{2s}\cdot,\cdot\rangle$. Since $d(u,\mathcal{M})=o(1)$,  by Lemma \ref{infinitesimals} we have $$\epsilon:=\|r\|_{L^{\frac{2n}{n+2s}}(\mathbb{S}^n)}=o(1).$$
Since the stability inequality is conformal invariant, we may assume
$$u=c_{n,s}+\epsilon \widetilde{r_0}~~~(\text{otherwise consider}~~-u),$$ where $\widetilde{r_0}$ satisfy $\|\widetilde{r_0}\|_{L^{\frac{2n}{n+2s}}(\mathbb{S}^n)}=1$
and
$$\int_{\mathbb{S}^n}\widetilde{r_0}Y_{1,j}d\omega=0,~~j=1,\cdots,n+1.$$
Expanding $\widetilde{r_0}$ in terms of spherical harmonics $\widetilde{r_0}=\widetilde{r_{0,1}}Y_{0,1}+\sum_{l\geq 2}\sum_{m=1}^{N(n,l)}\widetilde{r_{l,m}}Y_{l,m}$ with $\widetilde{r_{l,m}}=\int_{\mathbb{S}^n}\widetilde{r_0}Y_{l,m}d\omega$,
then $\|\widetilde{r_0}\|^2_{L^2}=\widetilde{r_{0,1}}^2+\sum_{l=1}^{\infty}\sum_{m=1}^{N(n,l)}\widetilde{r_{l,m}}^2$. We distinguish between the case $\widetilde{r_{0,1}}^2\leq \frac{n}{n+2s}\|\widetilde{r_0}\|^2_{L^2(\mathbb{S}^n)}$ and the case $\widetilde{r}^2_{0,1}\geq \frac{n}{n+2s}\|\widetilde{r}_0\|^2_{{L^2(\mathbb{S}^n})}$.
\vskip0.1cm

If $\widetilde{r_{0,1}}^2\leq \frac{n}{n+2s}\|\widetilde{r_0}\|^2_{L^2(\mathbb{S}^n)}$, we write $u=c_{n,s}+\epsilon r_{0,1}+\epsilon(r_{0}-r_{0,1})$, by (\ref{spherical exp})
we have
\begin{align}\begin{split}\label{spherical exp 1}
& \left\||u|^{-\frac{4s}{n+2s}}u-\mathcal{P}_{2s}(u)\right\|_{L^{\frac{2n}{n-2s}}(\mathbb{S}^n)}\\
& \geq\epsilon\frac{\Gamma(\frac{n}{2}-s+1)}{\Gamma(\frac{n}{2}+s+1)}\left(\beta^{-\frac{4s}{n+2s}}-\frac{n/2-s+1}{n/2+s+1}\right)
\frac{\|\widetilde{r_0}-\widetilde{r_{0,1}}Y_{0,1}\|^2_{L^2(\mathbb{S}^n)}}{\|\widetilde{r_0}-\widetilde{r_{0,1}}Y_{0,1}\|_{L^{\frac{2n}{n+2s}}(\mathbb{S}^n)}}+o(\epsilon),
\end{split}\end{align}
where $\beta=1+\epsilon\frac{ r_{0,1}}{c_{n,s}}$ with $\beta^{-\frac{4s}{n+2s}}-\frac{n/2-s+1}{n/2+s+1}>\frac{s}{n/2+s+1}$ for sufficiently small $\epsilon$.

Using (\ref{spherical exp 1}), H\"{o}lder inequality
$$\|\widetilde{r_0}-\widetilde{r_{0,1}}Y_{0,1}\|_{L^2(\mathbb{S}^n)}\geq |\mathbb{S}^n|^{-\frac{s}{n}}\|\widetilde{r_0}-\widetilde{r_{0,1}}Y_{0,1}\|_{L^{\frac{2n}{n+2s}}(\mathbb{S}^n)}$$
 and the fact
 $$\|\widetilde{r_0}-\widetilde{r_{0,1}}Y_{0,1}\|^2_{L^2(\mathbb{S}^n)}=\|r_0\|_{L^2(\mathbb{S}^n)}^2-r_{0,1}^2\geq \frac{2s}{n+2s}\|\widetilde{r_0}\|_{L^2(\mathbb{S}^n)}^2\geq|\mathbb{S}^n|^{-\frac{2s}{n}}\frac{2s}{n+2s},$$
 we can obtain
\begin{align}\begin{split}\label{case 1}
& \left\||u|^{-\frac{4s}{n+2s}}u-\mathcal{P}_{2s}(u)\right\|_{L^{\frac{2n}{n-2s}}(\mathbb{S}^n)}\\
& \geq\epsilon\frac{\Gamma(\frac{n}{2}-s+1)}{\Gamma(\frac{n}{2}+s+1)}\frac{s}{n/2+s+1}
|\mathbb{S}^n|^{-\frac{2s}{n}}\sqrt{\frac{2s}{n+2s}}+o(\epsilon)
\end{split}\end{align}
for $\epsilon$ sufficiently small.

If $\widetilde{r}^2_{0,1}\geq \frac{n}{n+2s}\|\widetilde{r}_0\|^2_{{L^2(\mathbb{S}^n})}$, since
\begin{align*}
& \left\||u|^{-\frac{4s}{n+2s}}u-\mathcal{P}_{2s}(u)\right\|_{L^{\frac{2n}{n-2s}}(\mathbb{S}^n)}\\
& =\left\|| c_{n,s}+\epsilon \widetilde{r_0}|^{-\frac{4s}{n+2s}}( c_{n,s}+\epsilon \widetilde{r_0})-\mathcal{P}_{2s}( c_{n,s}+\epsilon \widetilde{r_0})\right\|_{L^{\frac{2n}{n-2s}}(\mathbb{S}^n)}\\
& \geq -\int_{\mathbb{S}^n}[| c_{n,s}+\epsilon\widetilde{ r_0}|^{-\frac{4s}{n+2s}}( c_{n,s}+\epsilon \widetilde{r_0})-\mathcal{P}_{2s}( c_{n,s}+\epsilon \widetilde{r_0})]\widetilde{r_0}d\omega.
\end{align*}
Again expanding $r_0$ in terms of spherical harmonics $\widetilde{r_0}=\widetilde{r_{0,1}}Y_{0,1}+\sum_{l\geq 2}\sum_{m=1}^{N(n,l)}\widetilde{r_{l,m}}Y_{l,m}$ with $\widetilde{r_{l,m}}=\int_{\mathbb{S}^n}r_0Y_{l,m}d\omega$,
by Taylor expansion we have
\begin{align*}
& -\int_{\mathbb{S}^n}[| c_{n,s}+\epsilon \widetilde{r_0}|^{-\frac{4s}{n+2s}}( c_{n,s}+\epsilon \widetilde{r_0})]\widetilde{r_0}d\omega\\
& =-( c_{n,s})^{\frac{n-2s}{n+2s}}\int_{\mathbb{S}^n}\widetilde{r_0}d\omega-\epsilon\frac{n-2s}{n+2s}( c_{n,s})^{-\frac{4s}{n+2s}}\int_{\mathbb{S}^n}\widetilde{r_0}^2d\omega+o(\epsilon),\\
& =-( c_{n,s})^{\frac{n-2s}{n+2s}}\int_{\mathbb{S}^n}\widetilde{r_0}d\omega-\epsilon\frac{n-2s}{n+2s}( c_{n,s})^{-\frac{4s}{n+2s}}\sum_{l\neq 1}\sum_{m=1}^{N(n,l)}\widetilde{r_{l,m}}^2+o(\epsilon).
\end{align*}
On the other hand,
\begin{align*}
& \int_{\mathbb{S}^n}\mathcal{P}_{2s}( c_{n,s}+\epsilon \widetilde{r_0})\widetilde{r_0}d\omega=c_{n,s}\frac{\Gamma(n/2-s)}{\Gamma(n/2+s)}\int_{\mathbb{S}^n}r_0d\omega+\epsilon\sum_{l\neq1}\frac{\Gamma(l+n/2-s)}{\Gamma(l+n/2+s)}\sum_{m=1}^{N(n,l)}\widetilde{r_{l,m}}^2.
\end{align*}
By the above estimates, we can obtain
\begin{align*}
& \left\||u|^{-\frac{4s}{n+2s}}u-\mathcal{P}_{2s}(u)\right\|_{L^{\frac{2n}{n-2s}}(\mathbb{S}^n)}\\
& \geq\epsilon\sum_{l\neq1}\left(\frac{\Gamma(l+n/2-s)}{\Gamma(l+n/2+s)}-\frac{n-2s}{n+2s}( c_{n,s})^{-\frac{4s}{n+2s}}\right)\sum_{m=1}^{N(n,l)}\widetilde{r_{l,m}}^2+o(\epsilon)\\
& =\epsilon \frac{\Gamma(n/2-s)}{\Gamma(n/2+s)}\frac{4s}{n+2s}\widetilde{r_{0,1}}^2+\epsilon\sum_{l\geq 2}\left(\frac{\Gamma(l+n/2-s)}{\Gamma(l+n/2+s)}-\frac{n-2s}{n+2s}( c_{n,s})^{-\frac{4s}{n+2s}}\right)\sum_{m=1}^{N(n,l)}\widetilde{r_{l,m}}^2\\
& \geq \epsilon \frac{\Gamma(n/2-s)}{\Gamma(n/2+s)}\frac{4s}{n+2s}\widetilde{r_{0,1}}^2-\epsilon\frac{n-2s}{n+2s}( c_{n,s})^{-\frac{4s}{n+2s}}\sum_{l\geq 2}\sum_{m=1}^{N(n,l)}\widetilde{r_{l,m}}^2.
\end{align*}
Since $\widetilde{r_{0,1}}^2\geq \frac{n}{n+2s}\|\widetilde{r_0}\|^2_{L^2(\mathbb{S}^n)}\geq \frac{n}{n+2s}|\mathbb{S}^n|^{-\frac{2s}{n}}$, then
\begin{align}\begin{split}\label{case 2}
& \left\||u|^{-\frac{4s}{n+2s}}u-\mathcal{P}_{2s}(u)\right\|_{L^{\frac{2n}{n-2s}}(\mathbb{S}^n)}\\
& \geq \epsilon\frac{\Gamma(n/2-s)}{\Gamma(n/2+s)}\frac{4s}{n+2s}\frac{n}{n+2s}\|\widetilde{r_0}\|_{L^2}-\epsilon\frac{n-2s}{n+2s}\frac{\Gamma(n/2-s)}{\Gamma(n/2+s)}\frac{2s}{n+2s}\|\widetilde{r_0}\|^2_{2}\\
& =\epsilon\frac{\Gamma(n/2-s)}{\Gamma(n/2+s)}\frac{2s}{n+2s}|\mathbb{S}^n|^{-\frac{2s}{n}}+o(\epsilon).
\end{split}\end{align}
Along with the estimates (\ref{case 1}), (\ref{case 2}) and Lemma \ref{infinitesimals}, we complete the proof of Lemma \ref{loc sta sph}.
\end{proof}

Now we are in the position to prove the stability for critical points of the HLS inequality, i.e. Theorem \ref{stability}.

\begin{proof}[The proof of Theorem \ref{stability}]
Assume that Theorem \ref{stability} is not true. Then we can find a sequence $\{f_k\}$ with $\frac{1}{2}\mathcal{S}_{s,n}^{\frac{n}{2s}}\leq\|f_k\|^{\frac{2n}{n+2s}}_{L^{\frac{2n}{n+2s}}(\mathbb{R}^n)}\leq \frac{3}{2}\mathcal{S}_{s,n}^{\frac{n}{2s}}$ such that
$$\frac{\left\||f_k|^{-\frac{4s}{n+2s}}f_k-(-\Delta)^{-s}f_k\right\|_{L^\frac{2n}{n-2s}(\mathbb{R}^n)}}{d(f_k,\mathcal{M}_{HLS})}\rightarrow 0,$$
as $k\rightarrow \infty$. If $d(f_k,\mathcal{M}_{HLS})\rightarrow 0$ as $k\rightarrow +\infty$, then it follows from Lemma \ref{loc sta sph} that
$$\frac{\left\||f_k|^{-\frac{4s}{n+2s}}f_k-(-\Delta)^{-s}f_k\right\|_{L^\frac{2n}{n-2s}(\mathbb{R}^n)}}{d(f_k,\mathcal{M}_{HLS})}\geq C_{n,s},$$ which is a contradiction.
If $d(f_k,\mathcal{M}_{HLS})\rightarrow C_0\neq 0$ as $k\rightarrow +\infty$, then $$\left\||f_k|^{-\frac{4s}{n+2s}}f-(-\Delta)^{-s}f\right\|_{L^\frac{2n}{n-2s}(\mathbb{R}^n)}\rightarrow 0.$$
By lemma \ref{deHLS}, there exists a subsequence of $\{f_k\}$ (still denote by $\{f_k\}$) such that $d(f_k,\mathcal{M}_{HLS})\rightarrow 0$, which is a contradiction. Then we accomplish the proof of Theorem \ref{stability}.
\end{proof}

Then by a dual method and the Sobolev and Hardy-Littlewood-Sobolev inequality, we are in the position to prove the stability for critical points of the Sobolev inequality, i.e. Theorem \ref{sob-stability}.

\begin{proof}[The proof of Theorem \ref{sob-stability}]
Let $f=|g|^{\frac{4s}{n-2s}}g$ and through the fractional Sobolev inequality, we can get
$$\int_{\mathbb{R}^n}|f|^{\frac{2n}{n+2s}}dx=\int_{\mathbb{R}^n}|g|^{\frac{2n}{n-2s}}dx\leq (\mathcal{S}_{s,n}^{-1})^{\frac{n}{n-2s}} \big(\int_{\mathbb{R}^n}|(-\Delta)^{\frac{s}{2}}g|^2dx\big)^{\frac{n}{n-2s}}.$$
This together with $\int_{\mathbb{R}^n}|(-\Delta)^{\frac{s}{2}}g|^2dx\leq \big(\frac{3}{2}\big)^{\frac{n-2s}{n}}\mathcal{S}_{s,n}^{\frac{n}{2s}}$ gives that
$$\int_{\mathbb{R}^n}|f|^{\frac{2n}{n+2s}}dx\leq \frac{3}{2}\mathcal{S}_{s,n}^{\frac{n}{2s}}.$$
If $\int_{\mathbb{R}^n}|f|^{\frac{2n}{n+2s}}dx\geq\big(\frac{1}{4}\big)^{\frac{n-2s}{n+2s}} \mathcal{S}_{s,n}^{\frac{n}{2s}}$, we can prove the stability for critical points of the Sobolev inequality though a duality argument. In fact, by the Sobolev inequality and Theorem \ref{stability} (see remark \ref{remark of thm}), we have
\begin{align*}
& \|(-\Delta)^{s}g-|g|^{\frac{4s}{n-2s}}g\|_{\dot H^{-s}(\mathbb{R}^n)}=\|(-\Delta)^{s/2}(g-(-\Delta)^{-s}f)\|_{L^2(\mathbb{R}^n)}\\
& \gtrsim  \|g-(-\Delta)^{-s}f\|_{L^{\frac{2n}{n-2s}}(\mathbb{R}^n)}=\||f|^{-\frac{4s}{n+2s}}f-(-\Delta)^{-s}f\|_{L^{\frac{2n}{n-2s}}(\mathbb{R}^n)}\\
& \gtrsim \inf_{h\in M_{HLS}}\|f-h\|_{L^\frac{2n}{n+2s}(\mathbb{R}^n)}.
\end{align*}
Then for any $\epsilon>0$, there exist a $h_0\in \mathcal{M}_{HLS}$ such that
$$\|(-\Delta)^{s}g-|g|^{\frac{4s}{n-2s}}g\|_{\dot H^{-s}(\mathbb{R}^n)}\gtrsim \|f-h_0\|_{L^\frac{2n}{n+2s}(\mathbb{R}^n)}-\epsilon,$$
therefor, using the Hardy-Littlewood-Sobolev inequality we can obtain
\begin{align*}
& \|(-\Delta)^{s}g-|g|^{\frac{4s}{n-2s}}g\|_{\dot H^{-s}(\mathbb{R}^n)}\gtrsim \|f-h_0\|_{L^\frac{2n}{n+2s}(\mathbb{R}^n)}-\epsilon+\|(-\Delta)^{s}g-|g|^{\frac{4s}{n-2s}}g\|_{\dot H^{-s}(\mathbb{R}^n)}\\
& \gtrsim \|(-\Delta)^{-s/2}(f-h_0)\|_{L^2(\mathbb{R}^n)}+\|(-\Delta)^{s/2}(g-(-\Delta)^{-s}f)\|_{L^2(\mathbb{R}^n)}-\epsilon\\
& \gtrsim \|g-(-\Delta)^{-s}h_0\|_{\dot H^{s}(\mathbb{R}^n)}-\epsilon \gtrsim \inf_{h\in \mathcal{M}_s}\|g-h\|_{\dot H^{s}(\mathbb{R}^n)}-\epsilon.
\end{align*}
If $\int_{\mathbb{R}^n}|f|^{\frac{2n}{n+2s}}dx\leq\big(\frac{1}{4}\big)^{\frac{n-2s}{n+2s}} \mathcal{S}_{s,n}^{\frac{n}{2s}}$, there holds
\begin{align*}
& \|(-\Delta)^{s}g-|g|^{\frac{4s}{n-2s}}g\|_{\dot H^{-s}(\mathbb{R}^n)}=\|(-\Delta)^{s/2}(g-(-\Delta)^{-s}f)\|_{L^2(\mathbb{R}^n)}\\
& \geq \|(-\Delta)^{s/2}g\|_{L^2}-\|(-\Delta)^{-s/2}f)\|_{L^2(\mathbb{R}^n)}\geq \|(-\Delta)^{s/2}g\|_{L^2}-\mathcal{S}_{s,n}^{-1/2}\|f\|_{L^{\frac{2n}{n+2s}}(\mathbb{R}^n)}\\
& \geq \big(\frac{1}{2}\big)^{\frac{n-2s}{2n}} \mathcal{S}_{s,n}^{\frac{n}{4s}}-\big(\frac{1}{4}\big)^{\frac{n-2s}{2n}} \mathcal{S}_{s,n}^{\frac{n}{4s}}.
\end{align*}
On the other hand
$$\inf_{h\in \mathcal{M}_s}\|g-h\|_{\dot H^{s}(\mathbb{R}^n)}\leq \|g\|_{\dot H^{s}(\mathbb{R}^n)}+\mathcal{S}_{s,n}^{\frac{n}{4s}}\leq (\big(\frac{3}{2}\big)^{\frac{n-2s}{2n}}+1)\mathcal{S}_{s,n}^{\frac{n}{4s}},$$
which implies $\|(-\Delta)^{s}g-|g|^{\frac{4s}{n-2s}}g\|_{\dot H^{-s}(\mathbb{R}^n)}\gtrsim \inf_{h\in \mathcal{M}_s}\|g-h\|_{\dot H^{s}(\mathbb{R}^n)}$.
Then we have completed the proof of Theorem \ref{sob-stability}.
\end{proof}

\section{Appendix}
In this section, we will prove some lemmas which were used in our proof of the main Theorem.
First we prove the following decomposition lemma.

\begin{lemma}\label{decompostition}
For any $g\in L^{\frac{2n}{n+2s}}(\mathbb{S}^n)$, with
$$\inf_{h\in \mathcal{M}}\|g-h\|^2_{L^{\frac{2n}{n+2s}}(\mathbb{S}^n)}\leq \frac{1}{2}\mathcal{S}_{s,n}^{\frac{n+2s}{2s}}$$
then there exists $\varphi\in \mathcal{M}$ such that
$$\inf_{h\in \mathcal{M}}\|g-h\|_{L^{\frac{2n}{n+2s}}(\mathbb{S}^n)}=\|g-\varphi\|_{L^{\frac{2n}{n+2s}}(\mathbb{S}^n)}$$
and there exists $\phi\in \mathcal{M} $ such that
$$\inf_{h\in \mathcal{M}}\langle\mathcal{P}_{2s}(g-h),g-h\rangle=\langle\mathcal{P}_{2s}(g-\phi),g-\phi\rangle.$$
Moreover, denote $g=\phi+r$ and $\phi=c_{n,s}J_{\Phi}^{\frac{n+2s}{2n}}$ or $-c_{n,s}J_{\Phi}^{\frac{n+2s}{2n}}$, where $\Phi$ is a conformal transformation on the sphere, then
$$r\bot \text{span}\{J_{\Phi}^{1/2}Y_{1,i}\circ \Phi,~~i=1,\cdots,n+1\}.$$
\end{lemma}

\begin{proof}
First we claim that the infimums can be obtained, and we only prove the case when the distance is $H^{-s}(\mathbb{S}^n)$ norm. Denote by
$$F(\xi)=\langle\mathcal{P}_{2s}(g-v_{\xi}),g-v_{\xi}\rangle,$$ where $v_{\xi}=\pm c_{n,s}\left(\frac{\sqrt{1-|\xi|^2}}{1-\xi\cdot \eta}\right)^{\frac{n+2s}{2}}$, then $F(\xi)$ is a continuous function on
$B^{n+1}=\{x\in \mathbb{R}^{n+1}:|x|<1\}$. Assume $\xi_{k}\in B^{n+1}$ such that
$$\lim_{k\rightarrow \infty}F(\xi_k)=\inf_{h\in \mathcal{M}}\langle\mathcal{P}_{2s}(g-h),g-h\rangle.$$
Since there exist an $\xi_0\in \mathbb{R}^{n+1}$ with $|\xi_0|\leq 1$ satisfying $\xi_k \rightarrow \xi_0$ (up to a subsequence), then if $|\xi_0|<1$, by the continuity of $F(c,\xi)$ we obtain
$$F(\xi_0)=\inf_{h\in \mathcal{M}}\langle\mathcal{P}_{2s}(g-h),g-h\rangle.$$
If $|\xi_0|=1$, since $$\inf_{h\in M_{HLS}}\langle\mathcal{P}_{2s}(g-h),g-h\rangle \leq \frac{1}{2}\mathcal{S}_{s,n}^{\frac{n}{2s}},$$
recalling that $c_{n,s}=\left(\frac{\Gamma(\frac{n}{2}+s)}{\Gamma(\frac{n}{2}-s)}\right)^{\frac{n+2s}{4s}}$ we have
\begin{align*}
&\frac{1}{2}\mathcal{S}_{s,n}^{\frac{n}{2s}}\geq  \lim_{k\rightarrow \infty}\int_{\mathbb{S}^n}\mathcal{P}_{2s}(g\mp c_{n,s}(\frac{\sqrt{1-|\xi_k|^2}}{1-\xi_k\cdot\omega})^{\frac{n-2s}{2}})(g\mp c_{n,s}(\frac{\sqrt{1-|\xi_k|^2}}{1-\xi_k\cdot\omega})^{\frac{n-2s}{2}})d\sigma_\xi\\
& =\langle\mathcal{P}_{2s}g,g\rangle+\mathcal{S}_{s,n}^{\frac{n}{2s}}\mp 2\lim_{k\rightarrow \infty} \langle\mathcal{P}_{2s}g,c_{n,s}(\frac{\sqrt{1-|\xi_k|^2}}{1-\xi_k\cdot\omega})^{\frac{n+2s}{2}}\rangle.\\
\end{align*}
Next let us prove
$$\int_{\mathbb{S}^n}\mathcal{P}_{2s}g(\omega)c_{n,s}(\frac{\sqrt{1-|\xi_k|^2}}{1-\xi_k\cdot\omega})^{\frac{n+2s}{2}}d\omega=0.$$
In fact for any $\epsilon>0$
$$\int_{\{\omega\cdot \xi_0<1-\epsilon\}}\mathcal{P}_{2s}g(\omega)c_{n,s}(\frac{\sqrt{1-|\xi_k|^2}}{1-\xi_k\cdot\omega})^{\frac{n+2s}{2}}d\omega\rightarrow 0, \text{as}~~k\rightarrow\infty$$
by Lebesgue Dominated Convergence Theorem.

On the other hand, by the H\"{o}lder inequality and the fact $\mathcal{P}_{2s}g(\omega)\in L^{\frac{2n}{n-2s}}(\mathbb{S}^n)$ we have
$$|\int_{\{\omega\cdot \xi_0\geq 1-\epsilon\}}\mathcal{P}_{2s}g(\omega)c_{n,s}(\frac{\sqrt{1-|\xi_k|^2}}{1-\xi_k\cdot\omega})^{\frac{n+2s}{2}}d\omega|
\leq \|\mathcal{P}_{2s}g(\omega)\chi_{\{\omega\cdot \xi_0\geq 1-\epsilon\}}\|_{L^{\frac{2n}{n-2s}}(\mathbb{S}^n)}\mathcal{S}_{s,n}^{\frac{n+2s}{4s}}\rightarrow 0,$$
as $\epsilon\rightarrow 0$,
which implies a contradiction $\frac{1}{2}\mathcal{S}_{s,n}^{\frac{n}{2s}}>\mathcal{S}_{s,n}^{\frac{n}{2s}}$ . Then we complete the proof of the claim.
\vskip0.1cm

We can assume that $\phi=v_{\xi_0}=c_{n,s}J_{\Phi}^{\frac{n+2s}{2n}}\in \mathcal{M} $ such that
$$\inf_{h\in \mathcal{M}}\langle\mathcal{P}_{2s}(g-h),g-h\rangle=\langle\mathcal{P}_{2s}(g-\phi),g-\phi\rangle,$$
then we have
$$r\bot \text{span}\{\partial_{\xi_i} v_{\xi}|_{\xi_0},~~i=1,\cdots,n+1 \}.$$
in the ``inner product" $\langle\mathcal{P}_{2s}\cdot,\cdot\rangle$. We have already proved that
$$\text{span}\{v_{\xi_0},\partial_\xi v_{c,\xi}|_{\xi_0},~~~i=1,\cdots,n+1 \}=\text{span} \{J_{\Phi}^{\frac{n+2s}{2n}}Y_{0,1}, J_{\Phi}^{\frac{n+2s}{2n}}Y_{1,i}\circ \Phi,~~i=1,\cdots,n+1\}$$
in \cite{CLT2} and together with the fact $v_{\xi_0}=cJ_{\Phi}^{\frac{n+2s}{2n}}Y_{0,1}$ for some constant $c$, $v_{\xi_0}\perp \partial_{\xi_i} v_{c,\xi}$ and $J_{\Phi}^{\frac{n+2s}{2n}}Y_{0,1}\perp J_{\Phi}^{\frac{n+2s}{2n}}Y_{1,i}$ for $i=1,\cdots,n+1$, we know that
$$\text{span}\{\partial_\xi v_{c,\xi}|_{\xi_0}i=1,\cdots,n+1 \}=\text{span} \{ J_{\Phi}^{\frac{n+2s}{2n}}Y_{1,i}\circ \Phi,~~i=1,\cdots,n+1\}$$
\end{proof}

\begin{lemma}\label{approxi}
If $u, ~u_n \in L^{\frac{2n}{n+2s}}(\mathbb{S}^n)$ with $u_n\rightarrow u$ in $L^{\frac{2n}{n+2s}}(\mathbb{S}^n)$, then
$$\left\||u_n|^{-\frac{4s}{n+2s}}u_n-\mathcal{P}_{2s}u_n\right\|_{L^{\frac{2n}{n-2s}}(\mathbb{S}^n)}
\rightarrow\left\||u|^{-\frac{4s}{n+2s}}u-\mathcal{P}_{2s}u\right\|_{L^{\frac{2n}{n-2s}}(\mathbb{S}^n)}.$$
\end{lemma}
\begin{proof}
Using the inequality $||a|^{p-1}a-|b|^{p-1}b|\lesssim |a-b|^{p}$ for all $a,b\in \mathbb{R}^n$ and $0\leq p\leq 1$, we have
$$\||u|^{-\frac{4s}{n+2s}}u-|u_n|^{-\frac{4s}{n+2s}}u_n\|_{L^{\frac{2n}{n-2s}}(\mathbb{S}^n)}\lesssim 2^{-\frac{4s}{n+2s}}\|u-u_n\|^{\frac{n-2s}{n+2s}}_{L^{\frac{2n}{n+2s}}(\mathbb{S}^n)}\rightarrow0,$$
On the other hand
\begin{align*}
& \|\mathcal{P}_{2s}(u-u_n)\|_{L^{\frac{2n}{n-2s}}(\mathbb{S}^n)}=\sup\limits_{\|g\|_{L^{\frac{2n}{n+2s}}(\mathbb{S}^n)}\leq 1}\langle\mathcal{P}_{2s}(u-u_n),g\rangle\\
& \leq
\sup\limits_{\|g\|_{L^{\frac{2n}{n+2s}}(\mathbb{S}^n)}\leq 1}\|\mathcal{P}_{s}(u-u_n)\|_{L^2(\mathbb{S}^n)}\|\mathcal{P}_{s}g\|_{L^2(\mathbb{S}^n)}\\
& \leq \mathcal{S}_{s,n}^{-1}\|u-u_n\|_{L^{\frac{2n}{n+2s}}(\mathbb{S}^n)}\rightarrow0.
\end{align*}
\end{proof}
Thus we have
$$\left\||u_n|^{-\frac{4s}{n+2s}}u_n-\mathcal{P}_{2s}u_n\right\|_{L^{\frac{2n}{n-2s}}(\mathbb{S}^n)}
\rightarrow\left\||u|^{-\frac{4s}{n+2s}}u-\mathcal{P}_{2s}u\right\|_{L^{\frac{2n}{n-2s}}(\mathbb{S}^n)}.$$

\begin{lemma}\label{lemma-holder-1d}
Let
\[
F(t)=|t|^{\beta-1}t,
\qquad 0<\beta<1.
\]
Then there exists a constant $C_\beta>0$ such that
\[
|F(a+b)-F(a)|
\leq
C_\beta |b|^\beta,
\qquad
\forall\, a,b\in \mathbb R.
\]
In particular,
\[
\left|
|a+b|^{-\frac{4s}{n+2s}}(a+b)
-
|a|^{-\frac{4s}{n+2s}}a
\right|
\leq
C_n |b|^{\frac{n-2s}{n+2s}}.
\]
\end{lemma}

\begin{proof}
We divide the proof into two cases.

\medskip

\noindent
\textbf{Case 1:} $|a|\leq 2|b|$.

Since
\[
|F(t)|=|t|^\beta,
\]
we have
\[
\begin{aligned}
|F(a+b)-F(a)|
&\leq |F(a+b)|+|F(a)| \\
&= |a+b|^\beta+|a|^\beta.
\end{aligned}
\]
Using
\[
|a+b|
\leq |a|+|b|
\leq 3|b|,
\]
it follows that
\[
|F(a+b)-F(a)|
\leq
3^\beta |b|^\beta+2^\beta |b|^\beta
\leq C |b|^\beta.
\]

\medskip

\noindent
\textbf{Case 2:} $|a|>2|b|$.

Define
\[
F(t)=|t|^{\beta-1}t.
\]
Since $F\in C^1(\mathbb R\setminus\{0\})$, by the mean value theorem there exists $\theta\in(0,1)$ such that
\[
F(a+b)-F(a)
=
F'(a+\theta b)b.
\]
A direct computation gives
\[
|F'(t)|=\beta |t|^{\beta-1}.
\]
Hence
\[
|F(a+b)-F(a)|
=
\beta |a+\theta b|^{\beta-1}|b|.
\]
Moreover,
\[
|a+\theta b|
\geq |a|-|\theta b|
\geq |a|-|b|
\geq \frac{|a|}{2}.
\]
Therefore,
\[
|F(a+b)-F(a)|
\leq
C |a|^{\beta-1}|b|.
\]
Since $\beta-1<0$ and $|a|>2|b|$, we obtain
\[
|a|^{\beta-1}
\leq
(2|b|)^{\beta-1}.
\]
Consequently,
\[
|F(a+b)-F(a)|
\leq
C |b|^{\beta-1}|b|
=
C |b|^\beta.
\]
Thus, we have accomplished the proof of Lemma \ref{lemma-holder-1d}.
\end{proof}

\bibliographystyle{amsalpha}

\end{document}